\documentclass{SCIYA2017enOL}
\usepackage{amsmath,amssymb,amsthm,bm,mathrsfs,graphicx}
\crefname{equation}{}{}
\crefname{lemma}{Lemma}{Lemmas}
\crefname{theorem}{Theorem}{Theorems}
\crefname{discr}{Discretization}{Discretizations}

\newcommand{\dual}[1]{\langle {#1} \rangle}
\newcommand{\Dual}[1]{\left\langle {#1} \right\rangle}

\newcommand{\nm}[1]{\lVert {#1} \rVert}
\newcommand{\Nm}[1]{\left\lVert {#1} \right\rVert}

\newcommand{\ssnm}[1]
{
  \left\vert\kern-0.25ex
  \left\vert\kern-0.25ex
  \left\vert
  {#1}
  \right\vert\kern-0.25ex
  \right\vert\kern-0.25ex
  \right\vert
}

\online
\begin{document}

\ensubject{fdsfd}

\ArticleType{ARTICLES}
\Year{2022}
\Month{}%
\Vol{}
\No{1}
\BeginPage{1} %
\DOI{10.1007/s11425-000-0000-0}
\ReceiveDate{May 9, 2021}
\AcceptDate{August 30, 2022}

\title[]{Numerical analysis of a Neumann boundary control problem with a
stochastic parabolic equation}{Numerical analysis of a Neumann boundary control problem with a
stochastic parabolic equation}

\author[1,2]{Qin Zhou}{{zhouqinmath@stu.scu.edu.cn}}
\author[1,$\ast$]{Binjie Li}{{libinjie@scu.edu.cn}}




\address[1]{School of Mathematics, Sichuan University, Chengdu {\rm610064}, China}
\address[2]{School of Mathematics, China West Normal University, Nanchong {\rm 637002}, China}

\abstract{
  This paper analyzes the discretization of a Neumann boundary control problem
  with a stochastic parabolic equation, where an additive noise occurs in the
  Neumann boundary condition. The convergence is established for general
  filtration, and the convergence rate $ O(\tau^{1/4-\epsilon} +
  h^{1/2-\epsilon}) $ is derived for the natural filtration of the $ Q $-Wiener
  process.}



\maketitle


\section{Introduction}
Let $ (\Omega, \mathcal F, \mathbb P) $ be a given complete probability space
with a normal filtration $ \mathbb F = \{\mathcal F_t\}_{t \geqslant 0 } $ (i.e.,
$ \mathbb F $ is right-continuous and $ \mathcal F_0 $ contains all $ \mathbb P
$-null sets in $ \mathcal F $). Assume that $ \mathcal O \subset \mathbb R^d $
($d=2,3$) is a convex polygonal domain with the boundary $ \Gamma $. Let $ W $
be a $ Q $-Wiener process in $ L^2(\Gamma) $ of the form
\[
  W(t) = \sum_{n=0}^\infty \sqrt{\lambda_n} \beta_n(t) \phi_n,
  \quad t \geqslant 0,
\]
where the following conditions hold: $ \{\lambda_n\}_{n=0}^\infty $ is a
sequence of strictly positive numbers satisfying $ \sum_{n=0}^\infty \lambda_n <
\infty $; $ \{\phi_n\}_{n=0}^\infty $ is an orthonormal basis of $ L^2(\Gamma) $;
$ \{\beta_n\}_{n=0}^\infty $ is a sequence of real valued Brownian motions
defined on $ (\Omega, \mathcal F, \mathbb P) $, mutually independent and $
\mathbb F $-adapted. We consider the following model problem:
\begin{equation}
  \label{eq:model}
  \begin{aligned}
    \min_{u \in U_\text{ad}} \mathcal J(u,y) & :=
    \frac12 \nm{y-y_d}_{L^2(\Omega;L^2(0,T; L^2(\mathcal O)))}^2 +
    \frac{\nu}2 \nm{u}_{L^2(\Omega;L^2(0,T;L^2(\Gamma)))}^2,
  \end{aligned}
\end{equation}
subject to the state equation
\begin{equation}
  \label{eq:state}
  \begin{cases}
    \frac{\partial}{\partial t} y(t,x) - \Delta y(t,x) + y(t,x) = 0,
    & 0 \leqslant t \leqslant T, \, x \in \mathcal O, \\
    \partial_{\bm n}y(t,x) = u(t,x) +
    \sum_{n=0}^\infty \lambda_n^{1/2}
    (\sigma(t) \phi_n)(x) \frac{\mathrm{d}\beta_n(t)}{\mathrm{d}t},
    & 0 \leqslant t \leqslant T, \, x \in \Gamma, \\
    y(0,x) = 0, & x \in \mathcal O,
  \end{cases}
\end{equation}
where $ 0 < \nu,T < \infty $, $ y_d $ and $ \sigma $ are two given processes,
and $ \bm n $ is the outward unit normal vector to $ \Gamma $. The above
admissible space $ U_\text{ad} $ is defined by
\[
  U_\text{ad} := \big\{
    u \in L_{\mathbb F}^2(\Omega;L^2(0,T;L^2(\Gamma))) \mid
    u_* \leqslant u \leqslant u^*
  \big\},
\]
where $  u_* < u^* $ are two given constants.


In the past four decades, a considerable number of papers have been published in
the field of optimal control problems of stochastic partial differential
equations; see~\cite{Bensoussan1983,Hu_Peng_1991,Zhou1993,Debussche2007,
Guatteri2011, Fuhrman2012,Du2013,Fuhrman2013,Guatteri2013,Lu2014,Swiech2020,
LuZhang2021} and the references therein. By now it is still a
highly active research area. However, the numerical analysis of stochastic
parabolic optimal control problems is quite rare. The numerical analysis of
stochastic parabolic optimal control problems consists of two main
challenges. The first is to derive convergence rate for rough data. To keep the
co-state $ \mathbb F $-adapted, the adjoint equation has a correction term $ z(t)
\, \mathrm{d}W(t) $. The process $ z $ is of low temporal regularity for rough
data, and this makes the convergence rate difficult to derive. Here we
especially emphasize the case that the filtration is not the natural one of the
Wiener process. Since in this case the martingale representation theorem can not
be utilized, generally there is essential difficulty in analyzing the adjoint
equation of the stochastic optimal control problems; see~\cite{Lu2014}. The
second is to design efficient fully implementable algorithms. Generally, the
discretization of the adjoint equation will be used for the construction of an
efficient algorithm for a parabolic optimal control problem. The adjoint
equation of a stochastic optimal control problem is a backward stochastic
parabolic equation. To solve a backward parabolic equation requires computing a
conditional expectation at each time step. However, computing conditional
expectations is a challenging problem; see~\cite{Longstaff2001,Bouchard2004,
Gobet2005,Bender2007} and the references therein.

We briefly summarize the numerical analysis of stochastic parabolic optimal
control problems in the literature as follows. Dunst and Prohl~\cite{Dunst2016}
analyzed a spatial semi-discretization of a forward-backward stochastic heat
equation, and, using the least squares Monte-Carlo method to compute the
conditional expectations, they constructed two fully implementable
algorithms. Levajkovi\'c et al.~\cite{Levajkovic2017} presented an approximation
framework for computing the solution of the stochastic linear quadratic control
problem on Hilbert spaces; however, the time variable is not discretized in this
framework. Recently, Prohl and Wang~\cite{Wang2020a, Wang2020b} analyzed
numerically the distributed optimal control problems of stochastic heat equation
driven by additive and multiplicative noises, and Li and Zhou~\cite{LiZhou2020}
analyzed a distributed optimal control problem of a stochastic
parabolic equation driven by multiplicative noise. Li and Xie~\cite{Li-Xie-2022-space,Li-Xie-2022-time} analyzed the discretizations of
general stochastic linear quadratic control problems. We also refer the reader to
\cite{Gong2017} and the references therein for the numerical analysis of optimal
control problems governed by stochastic differential equations, and refer the
reader to \cite{Youngmi2018,Gunzburger2011} and the references therein for the
numerical analysis of optimal control problems governed by random partial
differential equations. The above works of stochastic parabolic optimal control
problems generally require that $ \mathbb F $ is the natural filtration of the Wiener process,
and, to our best knowledge, no numerical result is available for the stochastic
parabolic optimal control problems with boundary noises.





Since the noise is additive and the state equation is linear, we essentially use
a backward parabolic equation parameterized by the argument $ \omega \in \Omega
$ as the adjoint equation, instead of using the backward stochastic parabolic
equation. This enables us to establish the convergence of the discrete
stochastic optimal control problem for general filtration, which is the main
novelty of our theoretical results. Furthermore, under the condition that $
\mathbb F $ is the natural filtration of $ W(\cdot) $, we mainly use the mild
solution theory of the backward stochastic parabolic equations to derive the
convergence rate $ O(\tau^{1/4-\epsilon} + h^{1/2-\epsilon}) $ for the rough
data $ y_d \in L_\mathbb F^2(\Omega;L^2(0,T;L^2(\mathcal O))) $ and
\[
  \sigma \in L_\mathbb F^2(\Omega;L^2(0,T;\mathcal L_2^0))
  \cap L^\infty(0,T;L^2(\Omega;\mathcal L_2^0)),
\]
where $ \epsilon $ is a sufficiently small number.
Our numerical analysis can be easily extended to the distributed optimal control
problems governed by linear stochastic parabolic equations with additive noise
and general filtration. When the filtration is the natural one of the $ Q
$-Wiener process, our analysis can remove the restrictions on the data
imposed in \cite{Wang2020a}.

The rest of this paper is organized as follows. In \cref{sec:pre} we introduce
some notations and the first-order optimality condition of problem
\cref{eq:model}. In \cref{sec:discre} we present a discrete stochastic optimal
control problem, and in \cref{sec:proof} we prove the convergence of this
discrete stochastic optimal control problem.
Finally, in \cref{sec:conclusion} we conclude this paper.

\section{Preliminaries}
\label{sec:pre}
\medskip\noindent{\bf Conventions}. For any random variable $ v $ defined on $
(\Omega, \mathcal F, \mathbb P) $, $ \mathbb Ev $ denotes the expectation of $ v
$, and $ \mathbb E_t v $ denotes the conditional expectation of $ v $ with
respect to $ \mathcal F_t $ for any $ t > 0 $. We use $ \dual{\cdot,\cdot}
_{\mathcal O} $ and $ \dual{\cdot,\cdot}_{\Gamma} $ to denote the inner products
of $ L^2(\mathcal O) $ and $ L^2(\Gamma) $, respectively, and use $ [\cdot,
\cdot] $ and $ \dual{\cdot, \cdot} $ to denote the inner products of $
L^2(\Omega; L^2(\mathcal O)) $ and $ L^2(\Omega;L^2(\Gamma)) $,
respectively. The notation $ I $ denotes the identity mapping. For any Hilbert
space $ X $, we simply use $ \ssnm{\cdot}_X $ to denote the norm of the Hilbert
space $ L^2(\Omega; X) $. For any two separable Hilbert spaces $ X_1 $ and $ X_2
$, $ \mathcal L(X_1, X_2) $ denotes the set of all bounded linear operators from
$ X_1 $ to $ X_2 $, and $ \mathcal L_2(X_1, X_2) $ denotes the set of all
Hilbert-Schmidt operators from $ X_1 $ to $ X_2 $. Define
\[ 
  L_\lambda^2 := \Big\{
    \sum_{n=0}^\infty c_n \sqrt{\lambda_n} \phi_n \Big|\,
    \sum_{n=0}^\infty c_n^2 < \infty
  \Big\}
\]
and endow this space with the inner product
\[
  \Big(
    \sum_{n=0}^\infty c_n \sqrt{\lambda_n} \phi_n, \,\,
    \sum_{n=0}^\infty d_n \sqrt{\lambda_n} \phi_n
  \Big)_{L_\lambda^2} := \sum_{n=0}^\infty  c_n d_n
\]
for all $ \sum_{n=0}^\infty c_n \sqrt{\lambda_n} \phi_n, \sum_{n=0}^\infty d_n
\sqrt{\lambda_n} \phi_n \in L_\lambda^2 $.  In particular,
$ \mathcal L_2(L_\lambda^2, L^2(\Gamma)) $ is abbreviated to $ \mathcal
L_2^0 $. For any separable Hilbert space $ X $, define
\begin{align*} 
  L_\mathbb F^2(\Omega;L^2(0,T;X)) & := \Big\{
    \varphi: [0,T] \times \Omega \to X \mid
    \text{$ \varphi $ is $ \mathbb F $-progressively measurable and} \\
    & \qquad\qquad\qquad\qquad\qquad
    \ssnm{\varphi}_{L^2(0,T;X)} < \infty
  \Big\},
\end{align*}
and let $ \mathcal E_\mathbb F $ be the $ L^2\!(\Omega;\!L^2(0,T;\!X))
$-orthogonal projection onto $ L_\mathbb F^2\!(\Omega;\!L^2(0,T;\!X)) $. In
addition, let $ L_\mathbb F^2(\Omega; C([0,T];X)) $ be the space of all $ X
$-valued and $ \mathbb F $-adapted continuous processes. By saying that $
\mathbb F $ is the natural filtration of $ W(\cdot) $, we mean that $ \mathbb F
$ is generated by $ W(\cdot) $ and augmented by
\[
  \{ \mathcal N \in \mathcal F \mid \mathbb P(\mathcal N) = 0 \}.
\]




\medskip\noindent{\bf Definition of $ \dot H^\gamma $}. Let $ A $ be the
realization of the partial differential operator $ \Delta - I $ in $
L^2(\mathcal O) $ with homogeneous Neumann boundary condition. More precisely,
\begin{align*} 
  & \text{Domain}(A) := \{
    v \in H^2(\mathcal O) \mid
    \partial_{\bm n} v = 0 \text{ on } \Gamma
  \} \\
  & \text{ and } Av := \Delta v - v \quad
  \text{for all } v \in \text{Domain}(A).
\end{align*}
For any $ \gamma \geqslant 0 $, define
\[
  \dot H^\gamma := \{
    (-A)^{-\gamma/2} v \mid
    v \in L^2(\mathcal O)
  \}
\]
and endow this space with the norm
\[ 
  \nm{v}_{\dot H^\gamma} :=
  \nm{(-A)^{\gamma/2} v}_{L^2(\mathcal O)}
  \quad \forall v \in \dot H^\gamma.
\]
For each $ \gamma > 0 $, we use $ \dot H^{-\gamma} $ to denote the dual space of
$ \dot H^\gamma $, and use $ \dual{\cdot,\cdot}_{\dot H^\gamma} $ to denote the
duality pairing between $ \dot H^{-\gamma} $ and $ \dot H^\gamma $. In the
sequel, we will mainly use the notation $ \dot H^0 $ instead of $ L^2(\mathcal O)
$ for convenience.

\medskip\noindent{\bf Definitions of $ S_0 $ and $ S_1 $}. Extend $ A $ as a
bounded linear operator from $ \dot H^0 $ to $ \dot H^{-2} $ by
\[ 
  \dual{Av, \varphi}_{\dot H^{2}} :=
  \dual{v, A\varphi}_{\mathcal O}
\]
for all $ v \in \dot H^0 $ and $ \varphi \in \dot H^{2} $. The operator $ A $
then generates an analytic semigroup $ \{e^{tA}\}_{t \geqslant 0} $ in $ \dot
H^{-2} $ (cf.~\cite[Theorems 2.5 and 2.8]{Yagi2010}), possessing the following
property (cf.~\cite[Theorem~6.13, Chapter~2] {Pazy1983}): for any $ -2 \leqslant
\beta \leqslant \gamma \leqslant 2 $ and $ t > 0 $,
\begin{equation} 
  \label{eq:etA}
  \nm{e^{tA}}_{\mathcal L(\dot H^\beta, \dot H^\gamma)}
  \leqslant C t^{(\beta-\gamma)/2},
\end{equation}
where $ C $ is a positive constant depending only on $ \beta $ and $ \gamma $.
For any $ g \in L^2(0, T;\dot H^{-2}) $, the equation
\begin{equation} 
  \begin{cases}
    y'(t) - A y(t) = g(t) \quad \forall 0 \leqslant t \leqslant T, \\
    y(0) = 0
  \end{cases}
\end{equation}
admits a unique mild solution
\begin{equation}
  \label{eq:S0-def}
  (S_0g)(t) := \int_0^t e^{(t-s)A} g(s) \, \mathrm{d}s,
  \quad 0 \leqslant t \leqslant T.
\end{equation}
Symmetrically, for any $ g \in L^2(0,T;\dot H^{-2}) $, the equation
\begin{equation} 
  \begin{cases}
    -z'(t) - A z(t) = g(t) \quad \forall 0 \leqslant t \leqslant T, \\
    z(T) = 0
  \end{cases}
\end{equation}
admits a unique mild solution
\begin{equation}
  \label{eq:S1-def}
  (S_1g)(t) := \int_t^T e^{(s-t)A} g(s) \, \mathrm{d}s,
  \quad 0 \leqslant t \leqslant T.
\end{equation}
For any $ f,g \in L^2(0,T;\dot H^{-1}) $, a routine energy argument yields
\begin{align} 
  \nm{S_0f}_{L^2(0,T;\dot H^1)}
  & \leqslant \nm{f}_{L^2(0,T;\dot H^{-1})},
  \label{eq:S0R-C} \\
  \nm{S_1g}_{L^2(0,T;\dot H^1)}
  & \leqslant \nm{g}_{L^2(0,T;\dot H^{-1})},
  \label{eq:S1-C}
\end{align}
and it is easily verified that
\begin{equation} 
  \label{eq:S0-S1}
  \int_0^T \dual{g(t), (S_0f)(t)}_{\dot H^1} \, \mathrm{d}t =
  \int_0^T \dual{f(t), (S_1g)(t)}_{\dot H^1} \, \mathrm{d}t.
\end{equation}


\medskip\noindent{\bf Definition of $ G $.} Assume that $ \sigma \in L_{\mathbb
F}^2(\Omega;L^2(0,T;\mathcal L_2^0)) $. Let $ G $ be the mild solution of the
stochastic evolution equation
\begin{equation} 
  \begin{cases}
    \mathrm{d}G(t) = A G(t) \, \mathrm{d}t +
    \mathcal R \sigma(t) \, \mathrm{d}W(t),
    \quad \forall 0 \leqslant t \leqslant T, \\
    G(0) = 0,
  \end{cases}
\end{equation}
where $ \mathcal R \in \mathcal L(L^2(\Gamma), \dot H^{-1/2-\epsilon}) $, $
0 < \epsilon \leqslant 1/2 $, is defined by
\begin{equation} 
  \label{eq:calR-def}
  \dual{\mathcal Rv, \varphi}_{
    \dot H^{1/2+\epsilon}
  } := \dual{v,\varphi}_\Gamma  \quad
  \text{
    for all $ v \in L^2(\Gamma) $ and
    $ \varphi \in \dot H^{1/2+\epsilon} $
  }.
\end{equation}
For any $ 0 \leqslant t \leqslant T $, we have
(cf.~\cite[Chapter~3]{Gawarecki2011} and \cite[Chapter 5]{Prato2014})
\begin{align}
  G(t) &= \int_0^t e^{(t-s)A} \mathcal R \sigma(s) \, \mathrm{d}W(s)
  \quad \mathbb P \text{-a.s.,}
  \label{eq:G-def} \\
  G(t) &= \int_0^t A G(s) \, \mathrm{d}s +
  \int_0^t \mathcal R \sigma(s) \, \mathrm{d}W(s)
  \quad\text{$\mathbb P$-a.s.}
  \label{eq:G-int}
\end{align}
If $ \sigma \in L_\mathbb F^2(\Omega;L^2(0,T;\mathcal L_2^0)) \cap L^\infty(0,T;
L^2(\Omega;\mathcal L_2^0)) $, then, for any $ 0 < t \leqslant T $ and $ 0
\leqslant \gamma <(1/2-\epsilon)/2 $,
\begin{small}
\begin{align*} 
  \ssnm{G(t)}_{\dot H^{2\gamma}}^2 &=
  \int_0^t \ssnm{
    e^{(t-s)A} \mathcal R \sigma(s)
  }_{
    \mathcal L_2(L_\lambda^2,\dot H^{2\gamma})
  }^2 \, \mathrm{d}s \\
  & \leqslant \int_0^t
  \nm{e^{(t-s)A}}_{
    \mathcal L(\dot H^{-1/2-\epsilon}, \dot H^{2\gamma})
  }^2 \nm{\mathcal R}_{
    \mathcal L(L^2(\Gamma), \dot H^{-1/2-\epsilon})
  }^2
  \ssnm{\sigma(s)}_{\mathcal L_2^0}^2 \, \mathrm{d}s \\
  & \leqslant C
  \int_0^t (t-s)^{-2\gamma-1/2-\epsilon}
  \nm{\mathcal R}_{
    \mathcal L(L^2(\Gamma), \dot H^{-1/2-\epsilon})
  }^2 \ssnm{\sigma(s)}_{\mathcal L_2^0}^2 \, \mathrm{d}s
  \quad\text{(by \cref{eq:etA})} \\
  & \leqslant C
  t^{1/2-2\gamma-\epsilon} \nm{\mathcal R}_{
    \mathcal L(L^2(\Gamma), \dot H^{-1/2-\epsilon})
  }^2 \nm{\sigma}_{
    L^\infty(0,T;L^2(\Omega;\mathcal L_2^0))
  }^2,
\end{align*}
\end{small}
so that by the fact $ \mathcal R \in \mathcal L(L^2(\Gamma), \dot
H^{-1/2-\epsilon}) $ we obtain
\begin{equation} 
  \label{eq:G-regu}
  \ssnm{G(t)}_{ \dot H^{2\gamma} }
  \leqslant C t^{1/4-\gamma-\epsilon/2}
  \nm{\sigma}_{L^\infty(0,T;L^2(\Omega;\mathcal L_2^0))}.
\end{equation}
The above $ C $ denotes a positive constant depending only on $ \epsilon $, $
\gamma $ and $ \mathcal O $, and its value may differ in different places. For
more theoretical results, we refer the reader to \cite{Prato1993}.

\begin{remark}
  By \cite[Theorem 1.5.1.2]{Grisvard1985}, we have that $ \dot H^{1/2+\epsilon}
  $ is continuously embedded into $ L^2(\Gamma) $, and so the above operator $
  \mathcal R $ is well-defined.
\end{remark}


\medskip\noindent{\bf First-order optimality condition.} We call $ \bar u \in
U_\text{ad} $ a solution to problem \cref{eq:model} if $ \bar u $ minimizes the
cost functional
\begin{equation*} 
  \mathcal J(u) := \frac12
  \ssnm{
    S_0 \mathcal R u + G - y_d
  }_{L^2(0,T;L^2(\mathcal O))}^2 +
  \frac\nu2 \ssnm{u}_{L^2(0,T;L^2(\Gamma))}^2,
  \quad u \in U_\text{ad}.
\end{equation*}
It is standard that problem \cref{eq:model} admits a unique solution $ \bar u
$. Moreover, by \cref{eq:S0-S1,eq:calR-def} we have the first-order optimality
condition
\begin{equation}
  \label{eq:optim-cond}
  \int_0^T \dual{\bar p + \nu\bar u, u-\bar u} \, \mathrm{d}t
  \geqslant 0 \quad \text{ for all } u \in U_\text{ad},
\end{equation}
where
\begin{align}
  \bar y &:= S_0 \mathcal R\bar u + G, \label{eq:optim-y} \\
  \bar p &:= S_1\big( \bar y-y_d \big). \label{eq:optim-p}
\end{align}
It follows that
\begin{equation}
  \label{eq:barp-baru}
  \bar u = \mathcal P_{[u_*, u^*]} \Big(
    -\frac1\nu  \mathcal E_\mathbb F( \operatorname{tr} \bar p )
  \Big),
\end{equation}
where $ \operatorname{tr} $ is the trace operator from $ \dot H^1 $ to $
L^2(\Gamma) $ and
\begin{equation}
  \label{eq:Proj}
  \mathcal P_{[u_*,u^*]}(r) :=
  \begin{cases}
    u_* & \text{ if } r < u_*, \\
    r & \text{ if } u_* \leqslant r \leqslant u^*, \\
    u^* & \text{ if } r > u^*.
  \end{cases}
\end{equation}

\begin{remark}
  Assume that $ \mathbb F $ is the natural filtration of $ W(\cdot) $. The usual
  first-order optimality condition of problem \cref{eq:model} is that
  (cf.~\cite{Bensoussan1983,Debussche2007,Du2013,Fuhrman2013,Guatteri2011,Guatteri2013})
  \[
    \int_0^T \dual{\bar p + \nu \bar u, u - \bar u} \, \mathrm{d}t
    \geqslant 0 \quad \text{for all } u \in U_\text{ad}.
  \]
  The above $ \bar p $ is the first component of the solution $ (\bar p, \bar z)
  $ to the backward stochastic parabolic equation
  \[
    \begin{cases}
      \mathrm{d}\bar p(t) =
      -(A\bar p + S_0\mathcal R\bar u + G - y_d)(t) \, \mathrm{d}t +
      \bar z(t) \, \mathrm{d}\widetilde W(t),
      \quad 0 \leqslant t \leqslant T, \\
      \bar p(T) = 0,
    \end{cases}
  \]
  where $ \widetilde W $ is a cylindrical Wiener process in $ L^2(\mathcal O) $
  defined later by \cref{eq:wtW}. Moreover, following the proof of
  \cref{lem:I-Ptau-u}, we can obtain
  \[
    \bar u \in L_\mathbb F^2(\Omega;C([0,T];H^{1/2}(\Gamma))),
  \]
  where $ H^{1/2}(\Gamma) $ is a standard fractional order Sobolev space on $
  \Gamma $. 
\end{remark}


\section{Discrete stochastic optimal control problem}
\label{sec:discre}
Let $ J > 0 $ be an integer and define $ t_j := j\tau $ for each $ 0 \leqslant j
\leqslant J $, where $ \tau := T/J $. Let $ \mathcal K_h $ be a conventional
conforming, shape regular and quasi-uniform triangulation of $ \mathcal O $
consisting of $ d $-simplexes, and we use $ h $ to denote the maximum diameter
of the elements in $ \mathcal K_h $. Define
\begin{align*} 
  \mathcal V_h &:= \left\{
    v_h \in C(\overline{\mathcal O}) \mid\,
    v_h \text{ is linear on each }
    K \in \mathcal K_h
  \right\}, \\
  \mathcal X_{h,\tau} &:= \big\{
    V : [0,T] \times \Omega \to \mathcal V_h \mid\,
    V(t_j) \in L^2(\Omega;\mathcal V_h) \text{ and $ V $ is constant}  \\
    & \qquad\qquad\qquad\qquad\qquad\qquad
    \text{on }[t_j,t_{j+1}) \quad \text{for each } 0 \leqslant j < J
  \big\}.
\end{align*}
For any $ V \in \mathcal X_{h,\tau} $, by $ V_j $ we mean $ V(t_j) $ for each $
0 \leqslant j \leqslant J $.  Define $ Q_h: \dot H^{-1} \to \mathcal V_h $ by
\[ 
  \dual{Q_h v, v_h}_{\mathcal O} =
  \dual{v, v_h}_{\dot H^1}
  \quad \text{for all $ v \in \dot H^{-1} $ and $ v_h \in \mathcal V_h $},
\]
and define $ A_h: \mathcal V_h \to \mathcal V_h $ by
\[ 
  \dual{A_h v_h, w_h}_{\mathcal O} =
  -\int_{\mathcal O}\nabla v_h \cdot \nabla w_h + v_h w_h
  \quad\text{for all } v_h, w_h \in \mathcal V_h.
\]
For any $ g \in L^2(\Omega;L^2(0,T;\dot H^{-1})) $, define $ S_0^{h,\tau} g \in
\mathcal X_{h,\tau} $, an approximation of $ S_0g $, by
\begin{equation} 
  \label{eq:calS0}
  \begin{cases}
    (S_0^{h,\tau}g)_0 = 0, \\
    (S_0^{h,\tau}g)_{j+1} - (S_0^{h,\tau}g)_j =
    \tau A_h (S_0^{h,\tau}g)_{j+1} +
    \int_{t_j}^{t_{j+1}} Q_h g(t) \, \mathrm{d}t,
    \quad 0 \leqslant j < J.
  \end{cases}
\end{equation}
Define $ G_{h,\tau} \in \mathcal X_{h,\tau} $, an approximation of $ G $, by
\begin{equation} 
  \label{eq:calGhtau}
  \begin{cases}
    (G_{h,\tau})_0 = 0, \\
    (\!G_{h,\tau}\!)_{j+1} \!-\! (\!G_{h,\tau}\!)_j=
    \tau \! A_h (G_{h,\tau})_{j+1} \!+\! \int_{t_j}^{t_{j+1}}
    \!Q_h\!\mathcal R\, \sigma(t) \,\mathrm{d}W\!(t),
    \, 0 \leqslant j \!<\! J.
  \end{cases}
\end{equation}
\begin{remark}
  For any $ g \in L^2(\Omega;L^2(0,T;L^2(\Gamma))) $, a routine energy argument
  (cf.~\cite[Chapter~12]{Thomee2006} and \cite{Vexler2008I}) yields that
  \begin{equation} 
    \label{eq:S0-stab}
    \max_{0 \leqslant j \leqslant J}
    \ssnm{(S_0^{h,\tau}\mathcal Rg)_j}_{\dot H^0}
    \leqslant C \ssnm{g}_{L^2(0,T;L^2(\Gamma))},
  \end{equation}
  where $ C $ is a positive constant depending only on $ \mathcal O $.
\end{remark}

\begin{remark}
  Although $ \mathcal X_{h,\tau} \not\subset L_\mathbb F^2(\Omega;L^2(0,T;\dot
  H^0)) $, we have
  \[
    G_{h,\tau} \in L_\mathbb F^2(\Omega;L^2(0,T;\mathcal V_h)).
  \]
  Moreover, $ S_0^{h,\tau}g \in L_\mathbb F^2(\Omega;L^2(0,T;\mathcal V_h)) $
  for all $ g \in L_\mathbb F^2(\Omega;L^2(0,T;\dot H^{-1})) $.
\end{remark}

The discrete stochastic optimal control problem is to seek a discrete control $
\bar U \in U_\text{ad}^{h,\tau} $ such that
\begin{equation} 
  \label{eq:discretization}
  \mathcal J_{h,\tau}(\bar U) =
  \min_{U \in U_\text{ad}^{h,\tau}}
  \mathcal J_{h,\tau}(U),
\end{equation}
where the discrete admissible control space is defined by
\[ 
  U_\text{ad}^{h,\tau} := \left\{
    U \in U_\text{ad} \mid
    \text{$ U $ is constant on } [t_j,t_{j+1})
    \quad\forall 0 \leqslant j < J
  \right\}
\]
and the discrete cost functional is defined by
\begin{small}
\[ 
  \mathcal J_{h,\tau}(U) :=
  \frac12 \ssnm{
    S_0^{h,\tau} \mathcal R U + G_{h,\tau} -y_d
  }_{L^2(0,T;\dot H^0)}^2 +
  \frac\nu2 \ssnm{U}_{L^2(0,T;L^2(\Gamma))}^2
  \, \forall U \in U_\text{ad}^{h,\tau}.
\]
\end{small}

\begin{remark}
  Note that the discrete control is not discretized in space. This is just for
  the sake of numerical analysis, and there is no essential difficulty in
  extending the numerical analysis of this paper to the case that the discrete
  control is discretized by the standard continuous piecewise linear element
  method in space.
\end{remark}

Then, let us present the first-order optimality condition for the above discrete
problem. For any $ g \in L^2(\Omega;L^2(0,T;\dot H^{-1})) $, define $ S_1^{h,\tau}g
\in \mathcal X_{h,\tau} $ by
\begin{equation} 
  \label{eq:calS1}
  \begin{cases}
    (S_1^{h,\tau}g)_J = 0, \\
    (S_1^{h,\tau}g)_j - (S_1^{h,\tau}g)_{j+1} =
    \tau A_h (S_1^{h,\tau}g)_j +
    \int_{t_j}^{t_{j+1}} Q_h g(t) \, \mathrm{d}t,
    \quad 0 \leqslant j < J.
  \end{cases}
\end{equation}
In view of \cref{eq:calS0,eq:calS1}, a straightforward calculation yields that,
for any $ f \in L^2(\Omega;L^2(0,T; L^2(\Gamma))) $ and $ g \in L^2(\Omega;
L^2(0,T;\dot
H^0)) $,
\begin{equation} 
  \label{eq:S1htau-S2htau}
  \int_0^T [ S_0^{h,\tau} \mathcal R f, g ] \, \mathrm{d}t =
  \sum_{j=0}^{J-1} \int_{t_j}^{t_{j+1}}
  \dual{ f(t), (S_1^{h,\tau}g)_{j+1} } \, \mathrm{d}t.
\end{equation}
Using this equality, we readily conclude that
\begin{equation}
  \label{eq:discr-optim}
  \sum_{j=0}^{J-1} \int_{t_j}^{t_{j+1}}
  \dual{\bar P_{j+1} + \nu \bar U, U - \bar U} \, \mathrm{d}t
  \geqslant 0 \quad \text{for all $ U \in U_\text{ad}^{h,\tau} $,}
\end{equation}
where $ \bar U $ is the unique solution of problem \cref{eq:discretization} and
\begin{equation}
  \label{eq:barP-def}
  \bar P := S_1^{h,\tau}(S_0^{h,\tau}\mathcal R\bar U + G_{h,\tau} - y_d).
\end{equation}
It is evident from \cref{eq:discr-optim} that, for any $ 0 \leqslant j < J $,
\[
  \bar U_j = \mathcal P_{[u_*, u^*]} \Big(
    -\frac1\nu\mathbb E_{t_j} \operatorname{tr}\bar P_{j+1})
  \Big).
\]

\begin{remark}
  Let $ \bar U $ be the solution to problem \cref{eq:discretization}, and let
  \begin{align*}
    \bar Y &:= S_0^{h,\tau}\mathcal R \bar U + G_{h,\tau}, \\
    \bar P &:= S_1^{h,\tau}(\bar Y - y_d).
  \end{align*}
  Assume that there exist deterministic functions $ \mathscr U $ and $ \mathscr
  P $ such that
  \begin{equation}
    \label{eq:eq-in-remark}
    \bar U = \mathscr U(\bar Y) \quad\text{and} \quad
    \bar P = \mathscr P(\bar Y).
  \end{equation}
  Assume that $ \sigma(t) = 0 $ for all $ 0 \leqslant t \leqslant T $. From the
  fact that $ \bar U(t_0) $ is deterministic, it follows that $ \bar Y(t_1) $ is
  deterministic. Hence, according to the assumption, $ \bar U(t_1) $ and $ \bar
  P(t_1) $ are deterministic. Similarly, $ \bar Y(t_j) $, $ \bar U(t_j) $ and $
  \bar P(t_j) $ are deterministic for all $ 2 \leqslant j \leqslant J $. It
  follows that $ \bar Y $, $ \bar U $ and $ \bar P $ are deterministic.
  However, in general this is evidently incorrect, since $ y_d $ is a stochastic
  process. The above example illustrates that, for our model problem, generally
  we can not expect the existence of deterministic functions $ \mathscr U $ and
  $ \mathscr P $ such that \cref{eq:eq-in-remark} holds. Since
  \cref{eq:eq-in-remark} is crucial for the application of the algorithms in
  \cite{Dunst2016} to problem \cref{eq:discretization}, these algorithms are not
  applicable. 
\end{remark}

Finally, we present the convergence of the discrete stochastic optimal control
problem \cref{eq:discretization}. For any $ v \in L_\mathbb F^2(\Omega;L^2(0,T;
X)) $ with $ X $ being a separable Hilbert space, define $ \mathcal P_\tau v $ by
\begin{equation} 
  (\mathcal P_\tau v)(t) =
    \frac1\tau \mathbb E_{t_j} \int_{t_j}^{t_{j+1}} v(s) \, \mathrm{d}s
\end{equation}
for all $ t_j \leqslant t < t_{j+1} $ with $ 0 \leqslant j < J $.
\begin{theorem} 
  \label{thm:conv}
  Assume that $ y_d \in L_{\mathbb F}^2(\Omega;L^2(0,T;\dot H^0)) $ and
  \[
    \sigma \in L_{\mathbb F}^2(\Omega;L^2(0,T;\mathcal L_2^0))
    \cap L^\infty(0,T;L^2(\Omega;\mathcal L_2^0)).
  \]
  Let $ \bar u $ be the solution of problem \cref{eq:model}, let $ \bar p $ be
  defined by \cref{eq:optim-p}, and let $ \bar U $ be the solution of problem
  \cref{eq:discretization}. Then
  \begin{small}
  \begin{equation}
    \label{eq:conv-1}
    \begin{aligned}
      & \ssnm{\bar u - \bar U}_{L^2(0,T;L^2(\Gamma))} +
      \ssnm{
        S_0\mathcal R\bar u -
        S_0^{h,\tau}\mathcal R\bar U
      }_{L^2(0,T;\dot H^0)} \\
      \leqslant{} &
      C \Big(
        \tau^{1/4-\epsilon} + h^{1/2-\epsilon} +
        \ssnm{(I-\mathcal P_\tau)\bar u}_{L^2(0,T;L^2(\Gamma))} +
        \ssnm{(I-\mathcal P_\tau)\mathcal E_\mathbb F\bar p}_{
          L^2(0,T;L^2(\Gamma))
        }
      \Big).
    \end{aligned}
  \end{equation}
  \end{small}
  Moreover, if $ \mathbb F $ is the natural filtration of $ W(\cdot) $, then
  \begin{equation} 
    \label{eq:conv-2}
    \begin{aligned}
      & \ssnm{\bar u - \bar U}_{
        L^2(0,T;L^2(\Gamma))
      } + \ssnm{
        S_0\mathcal R\bar u -
        S_0^{h,\tau}\mathcal R\bar U
      }_{ L^2(0,T;\dot H^0) } \\
      \leqslant{} &
      C\big(
        \tau^{1/4-\epsilon} + h^{1/2-\epsilon}
      \big).
    \end{aligned}
  \end{equation}
  The above $ \epsilon $ is a sufficiently small positive number and the above $
  C $ is a positive constant depending only on $ u_* $, $ u^* $, $ y_d $, $
  \epsilon $, $ \sigma $, $ \nu $, $ \mathcal O $, $ T $ and the regularity
  parameters of $ \mathcal K_h $.
\end{theorem}


\section{Proofs}
\label{sec:proof}
Throughout this section, we will use the following conventions: $ \epsilon > 0 $
denotes a sufficiently small number, and its value may differ in different
places; $ a \lesssim b $ means $ a \leqslant Cb $, where $ C $ is a positive
constant depending only on $ \epsilon $, $ \mathcal O $, $ T $, the regularity
parameters of $ \mathcal K_h $ and the indexes of the Sobolev spaces where the
underlying functions belong to.


\subsection{Preliminary estimates}
For each $ \gamma \in \mathbb R $, let $ \dot H_h^\gamma $ be the space $
\mathcal V_h $ endowed with the norm
\[
  \nm{v_h}_{\dot H_h^\gamma} :=
  \nm{(-A_h)^{\gamma/2} v_h}_{\dot H^0}
  \quad \forall v_h \in \mathcal V_h.
\]
By the facts $ \mathcal R \in \mathcal L(L^2(\Gamma), \dot H^{-1/2-\epsilon}) $
and
\begin{equation}
  \label{eq:Qh-stab}
  \nm{Q_h}_{
    \mathcal L(\dot H^{-1/2-\epsilon}, \dot H_h^{-1/2-\epsilon})
  } \lesssim 1,
\end{equation}
we obtain
\begin{equation}
  \label{eq:Rh-stab}
  \nm{Q_h \mathcal R}_{
    \mathcal L(L^2(\Gamma), \dot H_h^{-(1/2+\epsilon)})
  } \lesssim 1.
\end{equation}
We also have the following two well-known estimates (cf.~\cite[Theorems 4.4.4,
4.5.11 and 5.7.6]{Brenner2008}):
\begin{align} 
    & \nm{I-Q_h}_{\mathcal L(\dot H^1, \dot H^0)}
    \lesssim h, \label{eq:lxy-1} \\
    & \nm{I-A_h^{-1}Q_h A}_{\mathcal L(\dot H^1, \dot H^0)}
    \lesssim h. \label{eq:lxy-2}
\end{align}
\begin{remark}
  For any $ v \in \dot H^{-1} $ we have
  \begin{align*} 
    \nm{Q_hv}_{\dot H_h^{-1}}
    & = \sup_{0 \neq \varphi_h \in \mathcal V_h}
    \frac{\dual{Q_hv, \varphi_h}_{\mathcal O}}{\nm{\varphi_h}_{\dot H_h^1}} \\
    & = \sup_{0 \neq \varphi_h \in \mathcal V_h}
    \frac{\dual{v, \varphi_h}_{\dot H^1}}{\nm{\varphi_h}_{\dot H_h^1}}
    \quad \text{(by the definition of $ Q_h $)} \\
    & \leqslant \sup_{0 \neq \varphi_h \in \mathcal V_h}
    \frac{\nm{v}_{\dot H^{-1}}
    \nm{\varphi_h}_{\dot H^1}}{\nm{\varphi_h}_{\dot H_h^1}}.
  \end{align*}
  Noting the fact $ \nm{\varphi_h}_{\dot H^1} = \nm{\varphi_h}_{\dot H_h^1} $,
  $ \forall \varphi_h \in \mathcal V_h $, we readily conclude that
  \[
    \nm{Q_h}_{\mathcal L(\dot H^{-1}, \dot H_h^{-1})}
    \leqslant 1.
  \]
  In addition, by definition, the restriction of $ Q_h $ to $ \dot H^0 $ is
  the $ L^2(\mathcal O) $-orthogonal projection onto $ \mathcal V_h $, and hence
  \[
    \nm{Q_h}_{\mathcal L(\dot H^0, \dot H_h^0)} \leqslant 1.
  \]
  Using the above two estimates, by interpolation (cf.~[27, Theorems 2.6 and
  4.36])
  we readily obtain \cref{eq:Qh-stab}.
\end{remark}

By the techniques in the proofs of \cite[Lemmas~3.2 and 7.3] {Thomee2006}, a
straightforward computation yields the following lemma.
\begin{lemma} 
  For any $ -2 \leqslant \beta \leqslant \gamma \leqslant 2 $ and $ t > 0 $, we
  have
  \begin{align} 
    \nm{ e^{tA_h} }_{
      \mathcal L(\dot H_h^{\beta}, \dot H_h^\gamma)
    } & \lesssim t^{(\beta-\gamma)/2}.
    \label{eq:etAh}
  \end{align}
  For any $ 0 \leqslant \beta \leqslant \gamma \leqslant 2 $ and $ t > 0 $,
  \begin{align}
    \nm{I- e^{tA_h}}_{
      \mathcal L(\dot H_h^\gamma, \dot H_h^\beta)
    } &\lesssim t^{(\gamma-\beta)/2}.
    \label{eq:I-etAh}
  \end{align}
  For any $ 0 \leqslant \gamma \leqslant 2 $ and $ j > 1 $,
  \begin{equation} 
    \label{eq:I-tauAh}
    \nm{
      (I-\tau A_h)^{-j}
    }_{
      \mathcal L(\dot H_h^{-\gamma}, \dot H_h^0)
    } < (j\tau)^{-\gamma/2}.
  \end{equation}
\end{lemma}


\begin{lemma} 
  \label{lemma:foo}
  Assume that $ 0 \leqslant \gamma \leqslant 1 $, $ 0 \leqslant k \leqslant j <
  J $ and $ t_k \leqslant t < t_{k+1} $. Then
  \begin{align} 
    \nm{
      e^{(t_{j+1} - t)A_h} - e^{(t_{j+1}-t_k)A_h}
    }_{
      \mathcal L(\dot H_h^{-\gamma}, \dot H_h^0)
    } \lesssim \tau^{(1-\gamma)/2} (t_{j+1} - t)^{-1/2}.
    \label{eq:foo-1}
  \end{align}
\end{lemma}
\begin{proof}
  We have
  \begin{align*}
    & \Nm{
      e^{(t_{j+1} - t)A_h} - e^{(t_{j+1} - t_k)A_h}
    }_{\mathcal L(\dot H_h^{-\gamma}, \dot H_h^0)} \\
    ={} &
    \Nm{
      (I-e^{(t-t_k)A_h}) e^{(t_{j+1}-t)A_h}
    }_{\mathcal L(\dot H_h^{-\gamma}, \dot H_h^0)} \\
    \leqslant{} &
    \nm{
      I - e^{(t-t_k)A_h}
    }_{
      \mathcal L(\dot H_h^{1-\gamma}, \dot H_h^0)
    } \nm{
      e^{(t_{j+1} - t)A_h}
    }_{\mathcal L(\dot H_h^{-\gamma}, \dot H_h^{1-\gamma})},
  \end{align*}
  and so \cref{eq:foo-1} follows from \cref{eq:etAh,eq:I-etAh}.
\end{proof}


\begin{lemma}
  For any $ 0 \leqslant \gamma \leqslant 2 $ and $ 1 \leqslant j \leqslant J $,
  \begin{equation} 
    \label{eq:etAh-err}
    \nm{
      e^{j \tau A_h} - (I-\tau A_h)^{-j}
    }_{
      \mathcal L(\dot H_h^{-\gamma}, \dot H_h^0)
    } \lesssim \tau^{-\gamma/2} j^{-1}.
  \end{equation}
\end{lemma}
\begin{proof}
  By \cite[Theorem~7.2]{Thomee2006} we have
  \[
    \Nm{
      e^{j\tau A_h} - (I-\tau A_h)^{-j}
    }_{
      \mathcal L(\dot H_h^0, \dot H_h^0)
    } \lesssim j^{-1}.
  \]
  Also, by \cref{eq:etAh,eq:I-tauAh} we have
  \begin{align*}
    \nm{ e^{j\tau A_h} - (I-\tau A_h)^{-j} }_{
      \mathcal L(\dot H_h^{-2}, \dot H_h^0)
    } \lesssim (j \tau)^{-1}.
  \end{align*}
  Hence, by interpolation (cf.~\cite[Theorems 2.6 and 4.36]{Lunardi2018}) we
  obtain \cref{eq:etAh-err}.
\end{proof}

\subsection{Convergence of \texorpdfstring{$ G_{h,\tau} $}{}}
Let $ G_h $ be the mild solution of the equation
\begin{equation}
  \label{eq:Gh-def}
  \begin{cases}
    \mathrm{d}G_h(t) = A_h G_h(t) \mathrm{d}t +
    Q_h \mathcal R \sigma(t) \, \mathrm{d}W(t),
    \quad 0 \leqslant t \leqslant T, \\
    G_h(0) = 0.
  \end{cases}
\end{equation}
Similar to \cref{eq:G-def}, we have, for each $ 0 \leqslant t \leqslant T $,
\begin{equation} 
  \label{eq:S0h-int}
  G_h(t) = \int_0^t e^{(t-s)A_h}
  Q_h \mathcal R \sigma(s) \, \mathrm{d}W(s)
  \quad \mathbb P \text{-a.s.}
\end{equation}

Let us first estimate $ \ssnm{G-G_h}_{L^2(0,T;\dot H^0)} $.
\begin{lemma}
  \label{lem:G-Gh}
  If
  \[
    \sigma \in L_\mathbb F^2(\Omega;L^2(0,T;\mathcal L_2^0)),
  \]
  then
  \begin{equation} 
    \label{eq:G-Gh}
    \ssnm{G - G_h}_{
      L^2(0,T;\dot H^0)
    } \lesssim h^{1/2-\epsilon} \ssnm{\sigma}_{
      L^2(0,T;\mathcal L^0_2)
    }.
  \end{equation}
\end{lemma}
\begin{proof}
  For any
  \[
    g \in L_\mathbb F^2(\Omega;L^2(0,T;\mathcal L_2(L_\lambda^2, \dot H^{-1}))),
  \]
  let $ \Pi g $ and $ \Pi_hg $ be the mild solutions of the equations
  \begin{equation*}
    \begin{cases}
      \mathrm{d}(\Pi g)(t) = A (\Pi g)(t) \mathrm{d}t +
      g(t) \, \mathrm{d}W(t),
      \quad 0 \leqslant t \leqslant T, \\
      (\Pi g)(0) = 0
    \end{cases}
  \end{equation*}
  and
  \begin{equation}
    \label{eq:Pih-def}
    \begin{cases}
      \mathrm{d}(\Pi_hg)(t) = A_h (\Pi_hg)(t) \mathrm{d}t +
      Q_hg(t) \, \mathrm{d}W(t),
      \quad 0 \leqslant t \leqslant T, \\
      (\Pi_hg)(0) = 0,
    \end{cases}
  \end{equation}
  respectively. By the celebrated It\^o's formula, we have the following
  standard estimates:
  \begin{align}
    \ssnm{\Pi g}_{L^2(0,T;\dot H^{\alpha+1})} \lesssim
    \ssnm{g}_{L^2(0,T;\mathcal L_2(L_\lambda^2,\dot H^\alpha))},
    \label{eq:jm-1} \\
    \ssnm{\Pi_h g}_{L^2(0,T;\dot H_h^{\alpha+1})} \lesssim
    \ssnm{Q_hg}_{L^2(0,T;\mathcal L_2(L_\lambda^2,\dot H_h^\alpha))},
    \label{eq:jm-2}
  \end{align}
  for any
  \[
    g \in L_\mathbb F^2(\Omega;L^2(0,T;\mathcal L_2(L_\lambda^2,\dot
    H^\alpha))) \quad \text{ with } \alpha \geqslant -1.
  \]
  Using the above estimates with $ \alpha=-1 $ and the fact
  \[
    \nm{Q_h}_{\mathcal L(\dot H^{-1}, \dot H_h^{-1})}
    \leqslant 1,
  \]
  we then obtain
  \begin{equation}
    \label{eq:zq-1}
    \nm{\Pi - \Pi_h}_{
      \mathcal L(
      L_\mathbb F^2(\Omega;L^2(0,T;\mathcal L_2(L_\lambda^2,\dot
      H^{-1}))), \,
      \dot H^0)
    } \lesssim 1.
  \end{equation}
  Let $ g \in L_\mathbb F^2(\Omega;L^2(0,T;\mathcal L_2(L_\lambda^2,
  \dot H^0))) $. For any
  $ 0 \leqslant t \leqslant T $ we have (cf.~\cite[Chapter 3]{Gawarecki2011})
  \[ 
    (\Pi g)(t) = \int_0^t A (\Pi g)(s) \, \mathrm{d}s +
    \int_0^t g(s) \, \mathrm{d}W(s)
    \quad \mathbb P \text{-a.s.,}
  \]
  and so
  \[ 
    (Q_h\Pi g)(t) = \int_0^t Q_h A (\Pi g)(s) \, \mathrm{d}s +
    \int_0^t Q_h g(s) \, \mathrm{d}W(s)
    \quad \mathbb P \text{-a.s.}
  \]
  It follows that
  \begin{equation} 
    \label{eq:QhG}
    \begin{cases}
      \mathrm{d}Q_h (\Pi g)(t) = Q_h A (\Pi g)(t) \, \mathrm{d}t +
      Q_hg (t) \, \mathrm{d}W(t),
      \quad 0 \leqslant t \leqslant T, \\
      Q_h (\Pi g)(0) = 0.
    \end{cases}
  \end{equation}
  Letting $ e_h := \Pi_h g - Q_h \Pi g $, by \cref{eq:Pih-def,eq:QhG} we get
  \begin{equation} 
    \begin{cases}
      \mathrm{d}e_h(t) = A_h e_h(t) \, \mathrm{d}t +
      (A_h Q_h \Pi g - Q_h A\Pi g)(t) \, \mathrm{d}t,
      \quad 0 \leqslant t \leqslant T, \\
      e_h(0) = 0.
    \end{cases}
  \end{equation}
  It follows that
  \begin{align*}
    \ssnm{e_h}_{L^2(0,T;\dot H_h^0)}
    & \lesssim
    \ssnm{
      (A_h Q_h  - Q_h A)\Pi g
    }_{L^2(0,T;\dot H_h^{-2})} \\
    & =
    \ssnm{
      (Q_h - A_h^{-1} Q_h A)\Pi g
    }_{L^2(0,T;\dot H_h^0)},
  \end{align*}
  and so
  \begin{align*}
    & \ssnm{(\Pi - \Pi_h)g}_{L^2(0,T;\dot H^0)} \\
    \leqslant{} &
    \ssnm{e_h}_{L^2(0,T;\dot H^0)} +
    \ssnm{(I-Q_h)\Pi g}_{L^2(0,T;\dot H^0)} \\
    \lesssim{} &
    \ssnm{(Q_h - A_h^{-1}Q_h A)\Pi g}_{L^2(0,T;\dot H^0)} +
    \ssnm{(I-Q_h)\Pi g}_{L^2(0,T;\dot H^0)} \\
    \lesssim{} &
    \ssnm{(I - A_h^{-1}Q_h A) \Pi g}_{L^2(0,T;\dot H^0)} +
    \ssnm{(I-Q_h)\Pi g}_{L^2(0,T;\dot H^0)} \\
    \lesssim{} &
    h \ssnm{\Pi g}_{L^2(0,T;\dot H^1)}
    \quad\text{(by \cref{eq:lxy-1,eq:lxy-2})} \\
    \lesssim{} & h \ssnm{g}_{L^2(0,T;\mathcal L_2(L_\lambda^2, \dot H^0))}
    \quad\text{(by \cref{eq:jm-1} with $ \alpha=0 $)}.
  \end{align*}
  This implies that
  \begin{equation}
    \label{eq:zq-2}
    \nm{\Pi - \Pi_h}_{
      \mathcal L(
        L_\mathbb F^2(\Omega;L^2(0,T;\mathcal L_2(L_\lambda^2, \dot H^0))),
        \, \dot H^0
      )
    } \lesssim h.
  \end{equation}
  In view of \cref{eq:zq-1}, \cref{eq:zq-2} and \cref{lem:auxi},
  by interpolation we obtain
  \[
    \nm{\Pi - \Pi_h}_{
      \mathcal L(
        L_\mathbb F^2(\Omega;L^2(0,T; \mathcal L_2(L_\lambda^2, \dot H^{-1/2-\epsilon}))),
        \, \dot H^0)
    } \lesssim h^{1/2-\epsilon}.
  \]
  Therefore, the desired estimate \cref{eq:G-Gh} follows from the following facts:
  \begin{align*} 
    & G = \Pi \mathcal R \sigma, \\
    & G_h = \Pi_h \mathcal R\sigma, \\
    & \mathcal R \in \mathcal L(L^2(\Gamma), \dot H^{-1/2-\epsilon}), \\
    & \sigma \in L_\mathbb F^2(\Omega;L^2(0,T;\mathcal L_2^0)).
  \end{align*}
  This completes the proof.
\end{proof}


Then let us estimate $ \ssnm{G_h - G_{h,\tau}}_{L^2(0,T;\dot H_h^0)} $.
\begin{lemma}
  \label{lem:Gh-Ghtau}
  If
  \[
    \sigma \in L_\mathbb F^2(\Omega;L^2(0,T;\mathcal L_2^0))
    \cap L^\infty(0,T;L^2(\Omega;\mathcal L_2^0)),
  \]
  then
  \begin{equation} 
    \label{eq:Gh-Ghtau}
    \max_{0 \leqslant j \leqslant J}
    \ssnm{(G_h - G_{h,\tau})(t_j)}_{\dot H_h^0}
    \lesssim \tau^{1/4-\epsilon} \,
    \nm{\sigma}_{L^\infty(0,T;L^2(\Omega;\mathcal L_2^0))}.
  \end{equation}
\end{lemma}
\begin{proof}
  Let $ 0 \leqslant j < J $ be arbitrary but fixed. By
  \cref{eq:calGhtau,eq:S0h-int} we obtain $ \mathbb P $-a.s.
  \[
    (G_h - G_{h,\tau})(t_{j+1}) =
    \sum_{k=0}^j \int_{t_k}^{t_{k+1}} \Big(
      e^{(t_{j+1} - t)A_h} -
      (I-\tau A_h)^{-(j-k+1)}
    \Big) Q_h \mathcal R\sigma(t) \, \mathrm{d}W(t),
  \]
  so that
  \begin{small}
  \begin{align*} 
    & \ssnm{
      (G_h - G_{h,\tau})(t_{j+1})
    }_{\dot H_h^0}^2 \\
    ={} &
    \sum_{k=0}^j \ssnm{
      \int_{t_k}^{t_{k+1}}
      \big(
        e^{(t_{j+1}-t)A_h} - (I-\tau A_h)^{-(j-k+1)}
      \big) Q_h \mathcal R\sigma(t)
      \, \mathrm{d}W(t)
    }_{\dot H_h^0}^2 \\
    ={} &
    \sum_{k=0}^j \int_{t_k}^{t_{k+1}}
    \ssnm{
      \big(
        e^{(t_{j+1} - t)A_h} - (I-\tau A_h)^{-(j-k+1)}
      \big) Q_h \mathcal R \sigma(t)
    }_{\mathcal L_2(L_\lambda^2,\dot H_h^0)}^2 \, \mathrm{d}t \\
    \leqslant{} &
    \sum_{k=0}^j \int_{t_k}^{t_{k+1}}
    \nm{
      e^{(t_{j+1}-t)A_h} - (I-\tau A_h)^{-(j-k+1)}
    }_{
      \mathcal L(\dot H_h^{-1/2-\epsilon}, \dot H_h^0)
    }^2 \\
    & \quad \times \nm{Q_h \mathcal R}_{
      \mathcal L(L^2(\Gamma), \dot H_h^{-1/2-\epsilon})
    }^2 \ssnm{\sigma(t)}_{\mathcal L_2^0}^2 \, \mathrm{d}t.
  \end{align*}
  \end{small}
  By \cref{eq:Rh-stab} we then conclude that
  \begin{small}
  \begin{equation}
    \label{eq:G-calG-1}
    \begin{aligned}
      & \ssnm{
        (G_h - G_{h,\tau})(t_{j+1})
      }_{ \dot H_h^0 }^2 \lesssim
      \nm{\sigma}_{L^\infty(0,T;L^2(\Omega;\mathcal L_2^0))}^2 \times {} \\
      & \qquad\qquad \sum_{k=0}^j \int_{t_k}^{t_{k+1}} \nm{
        e^{(t_{j+1} - t)A_h} - (I-\tau A_h)^{-(j-k+1)}
      }_{\mathcal L(\dot H_h^{-1/2-\epsilon}, \dot H_h^0)}^2
      \, \mathrm{d}t.
    \end{aligned}
  \end{equation}
  \end{small}
  By \cref{eq:etAh,eq:I-tauAh}, we have
  \[ 
    \int_{t_j}^{t_{j+1}} \nm{
      e^{(t_{j+1} - t)A_h} -
      (I-\tau A_h)^{-1}
    }_{\mathcal L(\dot H_h^{-1/2-\epsilon}, \dot H_h^0)}^2
    \, \mathrm{d}t
    \lesssim \tau^{1/2-\epsilon},
  \]
  and by \cref{eq:foo-1,eq:etAh-err} we have
  \begin{align*} 
    & \sum_{k=0}^{j-1} \int_{t_k}^{t_{k+1}}
    \nm{
      e^{(t_{j+1} - t)A_h} - (I-\tau A_h)^{-(j-k+1)}
    }_{
      \mathcal L(\dot H_h^{-1/2-\epsilon}, \dot H_h^0)
    }^2 \, \mathrm{d}t \\
    \lesssim{} &
    \sum_{k=0}^{j-1} \int_{t_k}^{t_{k+1}}
    \nm{
      e^{(t_{j+1} - t)A_h} - e^{(t_{j+1}-t_k)A_h}
    }_{
      \mathcal L(\dot H_h^{-1/2-\epsilon}, \dot H_h^0)
    }^2 \, \mathrm{d}t + {} \\
    & \qquad \sum_{k=0}^{j-1} \tau
    \nm{
      e^{(t_{j+1} - t_k)A_h} - (I-\tau A_h)^{-(j-k+1)}
    }_{
      \mathcal L(\dot H_h^{-1/2-\epsilon}, \dot H_h^0)
    }^2 \\
    \lesssim{} &
    \tau^{1/2-\epsilon} \sum_{k=0}^{j-1}
    \int_{t_k}^{t_{k+1}} (t_{j+1} - t)^{-1} \, \mathrm{d}t +
    \tau^{1/2-\epsilon} \sum_{k=0}^{j-1} (j-k+1)^{-1} \\
    \lesssim{} &
    \tau^{1/2-\epsilon} \big( \ln(1/\tau) + \ln(j+1) \big) \\
    \lesssim{} & \tau^{1/2-\epsilon}.
  \end{align*}
  Combining the above two estimates yields
  \[
    \sum_{k=0}^j \int_{t_k}^{t_{k+1}} \nm{
      e^{(t_{j+1} - t)A_h} - (I-\tau A_h)^{-(j-k+1)}
    }_{
      \mathcal L(\dot H_h^{-1/2-\epsilon}, \dot H_h^0)
    }^2 \lesssim \tau^{1/2-\epsilon}.
  \]
  Hence, by \cref{eq:G-calG-1} we get
  \[
    \ssnm{
      (G_h - G_{h,\tau})(t_{j+1})
    }_{ \dot H_h^0 } \lesssim \tau^{1/4-\epsilon}
    \nm{\sigma}_{L^\infty(0,T;L^2(\Omega;\mathcal L_2^0))}.
  \]
  Hence, as $ 0 \leqslant j < J $ is arbitrary, \cref{eq:Gh-Ghtau} follows from
  the fact
  \[
    (G_h - G_{h,\tau})(0) = 0.
  \]
  This completes the proof.
\end{proof}

\begin{lemma}
  \label{lem:Gh-Ghtau-l2}
  Under the condition of \cref{lem:Gh-Ghtau}, we have
  \begin{equation} 
    \label{eq:Gh-Ghtau-l2}
    \ssnm{G_h - G_{h,\tau}}_{L^2(0,T;\dot H_h^0)}
    \lesssim \tau^{1/4-\epsilon}
    \nm{\sigma}_{L^\infty(0,T;L^2(\Omega;\mathcal L_2^0))}.
  \end{equation}
\end{lemma}
\begin{proof}
  By \cref{eq:Gh-Ghtau}, we only need to prove
  \begin{equation} 
    \label{eq:Gh-Ghj}
    \sum_{j=0}^{J-1} \ssnm{
      G_h - G_h(t_j)
    }_{
      L^2(t_j,t_{j+1};\dot H_h^0)
    }^2 \lesssim \tau^{1/2-\epsilon}
    \nm{\sigma}_{L^\infty(0,T;L^2(\Omega;\mathcal L_2^0))}^2.
  \end{equation}
  To this end, we proceed as follows. For any $ t_j \leqslant t \leqslant t_{j+1}
  $ with $ 0 \leqslant j < J $, we have
  \[ 
    G_h(t) - e^{(t-t_j)A_h} G_h(t_j) =
    \int_{t_j}^t e^{(t-s)A_h} Q_h \mathcal R \sigma(s) \, \mathrm{d}W(s)
    \quad \mathbb P \text{-a.s.,}
  \]
  and so
  \begin{align*} 
    & \ssnm{
      G_h(t) - e^{(t-t_j)A_h} G_h(t_j)
    }_{\dot H_h^0}^2 \\
    ={} &
    \int_{t_j}^t \ssnm{
      e^{(t-s)A_h}Q_h \mathcal R\sigma(s)
    }_{
      \mathcal L_2(L_\lambda^2, \dot H_h^0)
    }^2 \, \mathrm{d}s \\
    \leqslant{} &
    \int_{t_j}^t \nm{e^{(t-s)A_h}}_{
      \mathcal L(\dot H_h^{-1/2-\epsilon}, \dot H_h^0)
    }^2 \nm{Q_h \mathcal R}_{
      \mathcal L(L^2(\Gamma), \dot H_h^{-1/2-\epsilon})
    }^2 \ssnm{\sigma(s)}_{\mathcal L_2^0}^2 \, \mathrm{d}s \\
    \lesssim{} &
    \int_{t_j}^t (t-s)^{-1/2-\epsilon} \, \mathrm{d}s
    \nm{\sigma}_{L^\infty(0,T;L^2(\Omega;\mathcal L_2^0))}^2
    \quad\text{(by \cref{eq:Rh-stab,eq:etAh})} \\
    \lesssim{} &
    (t-t_j)^{1/2-\epsilon} \nm{\sigma}_{
      L^\infty(0,T;L^2(\Omega;\mathcal L_2^0))
    }^2.
  \end{align*}
  It follows that
  \begin{equation}
    \label{eq:Gh-Ghj-3}
    \int_{t_j}^{t_{j+1}}
    \ssnm{
      G_h(t) - e^{(t-t_j)A_h} G_h(t_j)
    }_{\dot H_h^0}^2 \, \mathrm{d}t
    \lesssim \tau^{3/2-\epsilon} \nm{\sigma}_{
      L^\infty(0,T;L^2(\Omega;\mathcal L_2^0))
    }^2.
  \end{equation}
  Following the proof of \cref{eq:G-regu}, we obtain
  \[
    \ssnm{G_h(t_j)}_{\dot H_h^{1/2-\epsilon}}
    \lesssim \nm{\sigma}_{L^\infty(0,T;L^2(\Omega;\mathcal L_2^0))},
  \]
  so that
  \begin{align*} 
    & \int_{t_j}^{t_{j+1}} \ssnm{
      (I-e^{(t-t_j)A_h}) G_h(t_j)
    }_{\dot H_h^0}^2 \, \mathrm{d}t \\
    \lesssim{} &
    \int_{t_j}^{t_{j+1}} (t-t_j)^{1/2-\epsilon}
    \ssnm{G_h(t_j)}_{\dot H_h^{1/2-\epsilon}}^2 \, \mathrm{d}t
    \quad\text{(by \cref{eq:I-etAh})} \\
    \lesssim{} &
    \tau^{3/2-\epsilon} \nm{G_h(t_j)}_{
      \dot H_h^{1/2-\epsilon}
    }^2 \\
    \lesssim{} &
    \tau^{3/2-\epsilon} \nm{ \sigma }_{
      L^\infty(0,T;L^2(\Omega;\mathcal L_2^0))
    }^2.
  \end{align*}
  Combining the above estimate with \cref{eq:Gh-Ghj-3} yields
  \[ 
    \ssnm{G_h - G_h(t_j)}_{
      L^2(t_j,t_{j+1};\dot H_h^0)
    }^2 \lesssim \tau^{3/2-\epsilon}
    \nm{\sigma}_{L^\infty(0,T;L^2(\Omega;\mathcal L_2^0))}^2,
  \]
  and hence
  \begin{align*} 
    \sum_{j=0}^{J-1} \ssnm{
      G_h - G_h(t_j)
    }_{L^2(t_j,t_{j+1};\dot H_h^0)}^2
    &\lesssim \sum_{j=0}^{J-1} \tau^{3/2-\epsilon} \nm{\sigma}_{
      L^\infty(0,T;L^2(\Omega;\mathcal L_2^0))
    }^2 \\
    & \lesssim \tau^{1/2-\epsilon} \nm{\sigma}_{
      L^\infty(0,T;L^2(\Omega;\mathcal L_2^0))
    }^2.
  \end{align*}
  This proves \cref{eq:Gh-Ghj} and thus concludes the proof.
\end{proof}

Finally, combining \cref{lem:G-Gh,lem:Gh-Ghtau-l2}, we readily conclude the
following error estimate.
\begin{lemma} 
  \label{lem:G-Ghtau}
  If
  \[
    \sigma \in L_\mathbb F^2(\Omega;L^2(0,T;\mathcal L_2^0))
    \cap L^\infty(0,T;L^2(\Omega;\mathcal L_2^0)),
  \]
  then
  \begin{equation}
    \label{eq:G-Ghtau}
    \ssnm{G - G_{h,\tau}}_{
      L^2(0,T;\dot H^0)
    } \lesssim (\tau^{1/4-\epsilon} + h^{1/2-\epsilon})
    \nm{\sigma}_{L^\infty(0,T;L^2(\Omega;\mathcal L_2^0))}.
  \end{equation}
\end{lemma}

\subsection{Convergence of \texorpdfstring{$ S_0^{h,\tau} $ and $ S_1^{h,\tau} $}{}}
\begin{lemma} 
  For any $ g_h \in L^2(0,T;\mathcal V_h) $, define $ \{Y_j\}_{j=0} ^J \subset
  \mathcal V_h $ by
  \begin{equation} 
    \begin{cases}
      Y_0 = 0, \\
      Y_{j+1} - Y_j = \tau A_h Y_{j+1} + \int_{t_j}^{t_{j+1}} g_h(t) \,
      \mathrm{d}t, \quad 0 \leqslant j < J.
    \end{cases}
  \end{equation}
  Then, for any $ 0 \leqslant \gamma \leqslant 1 $,
  \begin{equation} 
    \label{eq:Yj-jmp}
    \Big(
      \sum_{j=0}^{J-1} \nm{Y_{j+1} - Y_j}_{\dot H_h^0}^2
    \Big)^{1/2} \lesssim \tau^{(1-\gamma)/2}
    \nm{g_h}_{L^2(0,T;\dot H_h^{-\gamma})}.
  \end{equation}
\end{lemma}
\begin{proof}
  Following the proof of \cite[Theorem 4.6]{Vexler2008I}, we obtain
  \begin{align} 
    \Big(
      \sum_{j=0}^{J-1} \nm{
        Y_{j+1} - Y_j
      }_{\dot H_h^0}^2
    \Big)^{1/2}
    \lesssim \tau^{1/2} \nm{g_h}_{L^2(0,T;\dot H_h^0)},
    \label{eq:Yj-jmp1} \\
    \big(
      \sum_{j=0}^{J-1} \tau \nm{ Y_j }_{\dot H_h^0}^2
    \big)^{1/2}
    \lesssim \nm{g_h}_{L^2(0,T;\dot H_h^{-2})}.
    \label{eq:YjH2}
  \end{align}
  From \cref{eq:YjH2} it follows that
  \begin{equation} 
    \label{eq:Yj-jmp2}
    \Big(
      \sum_{j=0}^{J-1} \nm{
        Y_{j+1} - Y_j
      }_{\dot H_h^0}^2
    \Big)^{1/2}
    \lesssim \tau^{-1/2} \nm{g_h}_{L^2(0,T;\dot H_h^{-2})}.
  \end{equation}
  In view of \cref{eq:Yj-jmp1,eq:Yj-jmp2}, by interpolation we obtain
  \cref{eq:Yj-jmp}. This completes the proof.
\end{proof}

\begin{lemma}
  \label{lem:S0-conv}
  For any $ g \in L_\mathbb F^2(\Omega;L^2(0,T;L^2(\Gamma))) $, we have
  \begin{equation} 
    \label{eq:S0-conv}
    \ssnm{(S_0 - S_0^{h,\tau})\mathcal Rg}_{L^2(0,T;\dot H^0)}
    \lesssim (\tau^{3/4-\epsilon} + h^{3/2-\epsilon})
    \ssnm{g}_{L^2(0,T;L^2(\Gamma))}.
  \end{equation}
\end{lemma}
\begin{proof}
  A routine energy argument (cf.~\cite[Chapter 12]{Thomee2006} and
  \cite[Theorems 5.1 and 5.5]{Vexler2008I}) gives that
  \begin{small}
  \[
    \Big(
      \sum_{j=0}^{J-1} \ssnm{
        S_0 \mathcal Rg - (S_0^{h,\tau}\mathcal Rg)_{j+1}
      }_{
        L^2(t_j,t_{j+1};\dot H^0)
      }^2
    \Big)^{1/2} \lesssim (\tau^{3/4-\epsilon} + h^{3/2-\epsilon})
    \ssnm{g}_{L^2(0,T;L^2(\Gamma))}.
  \]
  \end{small}
  We also have
  \begin{align*} 
    & \Big(
      \sum_{j=0}^{J-1} \ssnm{
        (S_0^{h,\tau}\mathcal Rg)_{j+1} -
        (S_0^{h,\tau}\mathcal Rg)_j
      }_{\dot H_h^0}^2
    \Big)^{1/2} \\
    \lesssim{} &
    \tau^{1/4-\epsilon}
    \ssnm{Q_h \mathcal Rg}_{L^2(0,T;\dot H_h^{-1/2-\epsilon})}
    \quad\text{(by \cref{eq:Yj-jmp})} \\
    \lesssim{} &
    \tau^{1/4-\epsilon} \ssnm{g}_{
      L^2(0,T;L^2(\Gamma))
    } \quad\text{(by \cref{eq:Rh-stab}).}
  \end{align*}
  Hence,
  \begin{small}
  \begin{align*}
    & \ssnm{(S_0 - S_0^{h,\tau})\mathcal Rg}_{L^2(0,T;\dot H^0)} \\
    ={} &
    \Big(
      \sum_{j=0}^{J-1} \ssnm{
        S_0\mathcal Rg - (S_0^{h,\tau}\mathcal Rg)_j
      }_{
        L^2(t_j,t_{j+1};\dot H^0)
      }^2
    \Big)^{1/2} \\
    \leqslant{} &
    \Big(
      \sum_{j=0}^{J-1} \ssnm{
        S_0\mathcal Rg - (S_0^{h,\tau}\mathcal Rg)_{j+1}
      }_{
        L^2(t_j,t_{j+1};\dot H^0)
      }^2
    \Big)^{1/2} + \Big(
      \sum_{j=0}^{J-1} \tau \ssnm{
        (S_0^{h,\tau}\mathcal Rg)_{j+1} -
        (S_0^{h,\tau}\mathcal Rg)_j
      }_{\dot H_h^0}^2
    \Big)^{1/2} \\
    \lesssim{} &
    (\tau^{3/4-\epsilon} + h^{3/2-\epsilon})
    \ssnm{g}_{L^2(0,T;L^2(\Gamma))}.
  \end{align*}
  \end{small}
  This proves \cref{eq:S0-conv} and thus concludes the proof.
\end{proof}

\begin{lemma}
  For any $ g \in L_\mathbb F^2(\Omega;L^2(0,T;\dot H^0)) $, we have
  \begin{small}
  \begin{equation}
    \label{eq:S1-calS1}
    \Big(
      \sum_{j=0}^{J-1} \ssnm{S_1g - (S_1^{h,\tau}g)_{j+1}}_{
        L^2(t_j,t_{j+1};L^2(\Gamma))
      }^2
    \Big)^{1/2} \lesssim (\tau^{3/4-\epsilon} + h^{3/2-\epsilon})
    \ssnm{g}_{L^2(0,T;\dot H^0)}.
  \end{equation}
  \end{small}
\end{lemma}
\noindent Since this lemma can be proved by using similar techniques as that in
the proof of \cref{lem:S0-conv}, its proof is omitted here.

\subsection{Proof of \texorpdfstring{\cref{thm:conv}}{}}
Let $ \{\varphi_n \mid n \in \mathbb N \} \subset \dot H^2 $ be an orthonormal
basis of $ \dot H^0 $ such that (cf.~\cite[Theorem~3.2.1.3]{Grisvard1985} and
\cite[Theorem~4.A and Section 4.5]{Zeidler1995})
\[ 
  -A \varphi_n = r_n \varphi_n \quad\text{for each } n \in \mathbb N,
\]
where $ \{r_n \mid n \in \mathbb N \} $ is a nondecreasing sequence of strictly
positive numbers with limit $ +\infty $. Let $ \widetilde W(t) $, $ t \geqslant
0 $, be a cylindrical Wiener process in $ L^2(\mathcal O) $ defined by
\begin{equation}
  \label{eq:wtW}
  \widetilde W(t)(v) = \sum_{n=0}^\infty
  \dual{v,\varphi_n}_{\mathcal O}
  \, \beta_n(t) \quad \text{for all $ t \geqslant 0 $ and $ v \in \dot H^0$}.
\end{equation}
Under the condition that $ \mathbb F $ is the natural filtration of $ W(\cdot) $,
for any $ g \in L_\mathbb F^2(\Omega; L^2(0,T;\dot H^0)) $, the backward
stochastic parabolic equation
\begin{equation} 
  \label{eq:p-z-def}
  \begin{cases}
    \mathrm{d}p(t) = -(Ap + g)(t) \, \mathrm{d}t +
    z(t) \, \mathrm{d}\widetilde W(t)
    \quad \forall 0 \leqslant t \leqslant T, \\
    p(T) = 0
  \end{cases}
\end{equation}
admits a unique strong solution $ (p,z) $, and
\begin{equation} 
  \label{eq:p-z-regu}
  \ssnm{p}_{L^2(0,T;\dot H^2)} +
  \ssnm{p}_{C([0,T];\dot H^1)} +
  \ssnm{z}_{L^2(0,T;\mathcal L_2(\dot H^0, \dot H^1))}
  \leqslant C \ssnm{g}_{L^2(0,T;\dot H^0)},
\end{equation}
where $ C $ is a positive constant independent of $ g $ and $ T $. Moreover,
for any $ 0 \leqslant s < t \leqslant T $,
\begin{equation} 
  \label{eq:p-z-int}
  p(s) - e^{(t-s)A} p(t) = \int_s^t
  e^{(r-s)A} g(r) \, \mathrm{d}r -
  \int_s^t e^{(r-s)A} z(r) \, \mathrm{d}\widetilde W(r)
  \quad \mathbb P \text{-a.s.}
\end{equation}
Using the fact $ p(T) = 0 $, by \cref{eq:p-z-int} we obtain, for any $ 0
\leqslant s \leqslant T $,
\begin{equation} 
  \label{eq:p-z-int2}
  p(s) = \int_s^T e^{(r-s)A} g(r) \, \mathrm{d}r -
  \int_s^T e^{(r-s)A} z(r) \, \mathrm{d}\widetilde W(r)
  \quad \mathbb P \text{-a.s.}
\end{equation}

\begin{remark}
  For the above theoretical results of equation \cref{eq:p-z-def}, we refer the
  reader to \cite{Hu_Peng_1991} and \cite[Chapter 4]{LuZhang2021}.
\end{remark}

\begin{lemma}
  \label{lem:I-Ptau-S1g}
  Assume that $ \mathbb F $ is the natural filtration of $ W(\cdot) $. For any $
  g \in L_\mathbb F^2(\Omega;L^2(0,T;\dot H^0)) $, we have
  \begin{equation}
    \label{eq:I-Ptau-S1g}
    \ssnm{(I-\mathcal P_\tau) \mathcal E_\mathbb F S_1g}_{
      L^2(0,T;L^2(\Gamma))
    } \lesssim \tau^{1/2} \ssnm{g}_{L^2(0,T;\dot H^0)}.
  \end{equation}
\end{lemma}
\begin{proof}
  Let $ (p,z) $ be the strong solution of equation \cref{eq:p-z-def}. Since
  \cref{eq:p-z-int2} implies, for each $ 0 \leqslant t \leqslant T $,
  \[ 
    p(t) = \mathbb E_t \Big(
      \int_t^T e^{(s-t)A} g(s) \, \mathrm{d}s
    \Big) \quad \mathbb P \text{-a.s.,}
  \]
  by \cref{eq:S1-def} we then obtain
  \begin{equation} 
    p = \mathcal E_\mathbb F S_1g.
  \end{equation}
  It remains, therefore, to prove
  \begin{equation}
    \label{eq:I-Ptau-S1g2}
    \ssnm{(I-\mathcal P_\tau)p}_{L^2(0,T;L^2(\Gamma))}
    \lesssim \tau^{1/2} \ssnm{g}_{L^2(0,T;\dot H^0)}.
  \end{equation}

  To this end, we proceed as follows. Let $ 1/2 < \gamma < 1 $ be arbitrary but
  fixed. For any $ t_j \leqslant t \leqslant t_{j+1} $ with $ 0 \leqslant j < J
  $, by \cref{eq:p-z-int} we have
  \[ 
    p(t_j) - e^{(t-t_j)A}p(t) =
    \int_{t_j}^t e^{(s-t_j)A} g(s) \, \mathrm{d}s -
    \int_{t_j}^t e^{(s-t_j)A} z(s) \, \mathrm{d}\widetilde W(s)
    \quad \mathbb P \text{-a.s.,}
  \]
  and so
  \begin{small}
  \begin{align*} 
    & \ssnm{
      p(t_j) - e^{(t-t_j)A} p(t)
    }_{\dot H^\gamma}^2 \\
    \leqslant{} &
    2 \ssnm{
      \int_{t_j}^t e^{(s-t_j)A} g(s) \, \mathrm{d}s
    }_{\dot H^\gamma}^2 + 2 \ssnm{
      \int_{t_j}^t e^{(s-t_j)A} z(s) \, \mathrm{d}\widetilde W(s)
    }_{\dot H^\gamma}^2 \\
    ={} &
    2 \ssnm{
      \int_{t_j}^t e^{(s-t_j)A} g(s) \, \mathrm{d}s
    }_{\dot H^\gamma}^2 + 2
    \int_{t_j}^t \ssnm{e^{(s-t_j)A} z(s)}_{
      \mathcal L_2(\dot H^0, \dot H^\gamma)
    }^2 \, \mathrm{d}s \\
    \leqslant{} &
    2 \Big(
      \int_{t_j}^t \ssnm{ e^{(s-t_j)A} g(s) }_{\dot H^\gamma} \, \mathrm{d}s
    \Big)^2 + 2 \int_{t_j}^t
    \ssnm{ e^{(s-t_j)A} z(s) }_{\mathcal L_2(\dot H^0, \dot H^\gamma)}^2 \, \mathrm{d}s \\
    \lesssim{} &
    \Big(
      \int_{t_j}^t (s-t_j)^{-\gamma/2} \ssnm{g(s)}_{\dot H^0}
      \, \mathrm{d}s
    \Big)^2 + \int_{t_j}^t \ssnm{z(s)}_{
      \mathcal L_2(\dot H^0, \dot H^\gamma)
    }^2 \, \mathrm{d}s \quad\text{(by \cref{eq:etA})} \\
    \lesssim{} &
    (t-t_j)^{1-\gamma} \ssnm{g}_{L^2(t_j,t;\dot H^0)}^2 +
    \ssnm{z}_{L^2(t_j,t; \mathcal L_2(\dot H^0, \dot H^\gamma))}^2.
  \end{align*}
  \end{small}
  It follows that, for any $ 0 \leqslant j < J $,
  \begin{small}
  \begin{align*}
    & \int_{t_j}^{t_{j+1}}
    \ssnm{p(t_j) - e^{(t-t_j)A}p(t)}_{\dot H^\gamma}^2
    \, \mathrm{d}t \\
    \lesssim{} &
    \tau^{2-\gamma} \ssnm{g}_{L^2(t_j,t_{j+1};\dot H^0)}^2 +
    \tau \ssnm{z}_{L^2(t_j,t_{j+1};\mathcal L_2(\dot H^0, \dot H^\gamma))}^2.
  \end{align*}
  \end{small}
  Hence,
  \begin{align*}
    {} &
    \sum_{j=0}^{J-1} \int_{t_j}^{t_{j+1}}
    \ssnm{p(t_j) - e^{(t-t_j)A} p(t)}_{\dot H^\gamma}^2 \, \mathrm{d}t \\
    \lesssim{} &
    \tau^{2-\gamma} \ssnm{g}_{L^2(0,T;\dot H^0)}^2 +
    \tau \ssnm{z}_{L^2(0,T;\mathcal L_2(\dot H^0, \dot H^\gamma))}^2.
  \end{align*}
  We also have
  \begin{align*} 
    & \sum_{j=0}^{J-1} \int_{t_j}^{t_{j+1}}
    \ssnm{(I-e^{(t-t_j)A}) p(t)}_{\dot H^\gamma}^2 \, \mathrm{d}t \\
    \lesssim{} &
    \sum_{j=0}^{J-1} \int_{t_j}^{t_{j+1}}
    (t-t_j)^{2-\gamma} \ssnm{p(t)}_{\dot H^2}^2 \, \mathrm{d}t \\
    \lesssim{} &
    \tau^{2-\gamma} \ssnm{p}_{L^2(0,T;\dot H^2)}^2,
  \end{align*}
  by the evident estimate (cf.~\cref{eq:I-etAh})
  \[
    \nm{I-e^{tA}}_{\mathcal L(\dot H^2, \dot H^\gamma)}
    \lesssim t^{1-\gamma/2}.
  \]
  Consequently,
  \begin{small} 
  \begin{align}
    & \sum_{j=0}^{J-1} \ssnm{p - p(t_j)}_{
      L^2(t_j,t_{j+1};\dot H^\gamma)
    }^2 \notag \\
    \lesssim{} &
    \sum_{j=0}^{J-1} \int_{t_j}^{t_{j+1}}
    \ssnm{p(t_j) - e^{(t-t_j)A}p(t)}_{\dot H^\gamma}^2 \, \mathrm{d}t +
    \sum_{j=0}^{J-1} \int_{t_j}^{t_{j+1}}
    \ssnm{(I-e^{(t-t_j)A})p(t)}_{\dot H^\gamma}^2 \, \mathrm{d}t \notag \\
    \lesssim{} & \tau^{2-\gamma} \ssnm{g}_{L^2(0,T;\dot H^0)}^2 +
    \tau \ssnm{z}_{L^2(0,T;\mathcal L_2(\dot H^0, \dot H^\gamma))}^2 +
    \tau^{2-\gamma} \ssnm{p}_{L^2(0,T;\dot H^2)}^2 \notag \\
    \lesssim{} &
    \tau \ssnm{g}_{L^2(0,T;\dot H^0)}^2,
    \label{eq:p-pj2}
  \end{align}
  \end{small}
  by \cref{eq:p-z-regu}. Therefore, from the inequality
  \[
    \ssnm{(I-\mathcal P_\tau)p}_{
      L^2(t_j,t_{j+1};\dot H^\gamma)
    } \leqslant \ssnm{p - p(t_j)}_{L^2(t_j,t_{j+1};\dot H^\gamma)}
    \quad \forall 0 \leqslant j < J,
  \]
  we conclude that
  \begin{equation} 
    \label{eq:p-pj}
    \ssnm{(I-\mathcal P_\tau)p}_{
      L^2(0,T;\dot H^\gamma)
    } \lesssim \tau^{1/2} \ssnm{g}_{L^2(0,T;\dot H^0)}.
  \end{equation}
  The desired estimate \cref{eq:I-Ptau-S1g2} then follows from the trace
  inequality (cf.~\cite[Theorem~1.5.1.2]{Grisvard1985})
  \[
    \nm{v}_{L^2(\Gamma)} \lesssim \nm{v}_{\dot H^{\gamma}}
    \quad \forall v \in \dot H^{\gamma}.
  \]
  This completes the proof.
\end{proof}

\begin{lemma}
  For any $ v \in L^2(\Omega;L^2(0,T;\dot H^1)) $, we have
  \begin{equation}
    \label{eq:trE}
    \mathcal E_\mathbb F \operatorname{tr} v =
    \operatorname{tr} \mathcal E_\mathbb Fv.
  \end{equation}
\end{lemma}
\begin{proof}
  Let $ \operatorname{tr}^*: L^2(\Gamma) \to \dot H^{-1} $ be the adjoint of $
  \operatorname{tr} $. By definition, $ \mathcal E_\mathbb F \operatorname{tr} v $ is the
  $ L^2(\Omega;L^2(0,T;L^2(\Gamma))) $-orthogonal projection of $
  \operatorname{tr} v $ onto $ L_\mathbb F^2(\Omega;L^2(0,T;L^2(\Gamma))) $, and
  $ \mathcal E_\mathbb F v $ is the $ L^2(\Omega;L^2(0,T;\dot H^1)) $-orthogonal
  projection of $ v $ onto $ L_\mathbb F^2(\Omega;L^2(0,T;\dot H^1)) $. For any
  $ \varphi \in L_\mathbb F^2(\Omega;L^2(0,T; L^2(\Gamma))) $, we have
  \begin{align*}
    \int_0^T \dual{
      \mathcal E_\mathbb F \operatorname{tr} v, \varphi
    } \, \mathrm{d}t & =
    \int_0^T \dual{
      \operatorname{tr}v, \varphi
    } \, \mathrm{d}t =
    \int_0^T \mathbb E \dual{\operatorname{tr}^* \varphi, v}_{\dot H^1}
    \, \mathrm{d}t \\
    &= \int_0^T \mathbb E \dual{
      \operatorname{tr}^*\varphi,
      \mathcal E_\mathbb Fv
    }_{\dot H^1} \, \mathrm{d}t =
    \int_0^T \dual{
      \operatorname{tr}\mathcal E_\mathbb Fv, \varphi
    } \, \mathrm{d}t,
  \end{align*}
  which implies \cref{eq:trE}.
\end{proof}

\begin{lemma}
  \label{lem:I-Ptau-u}
  Assume that $ \mathbb F $ is the natural filtration of $ W(\cdot) $.  For any
  $ g \in L_\mathbb F^2(\Omega;L^2(0,T;\dot H^0)) $, we have
  \begin{equation}
    \label{eq:I-Ptau-u}
    \ssnm{(I-\mathcal P_\tau)u}_{
      L^2(0,T;L^2(\Gamma))
    } \lesssim \tau^{1/2} \nu^{-1} \ssnm{g}_{L^2(0,T;\dot H^0)},
  \end{equation}
where
\[
  u := \mathcal P_{[u_*,u^*]}
  \big(-\nu^{-1} \mathcal E_\mathbb F(\operatorname{tr}S_1g) \big).
\]
\end{lemma}
\begin{proof}
  Let $ (p,z) $ be the solution of equation \cref{eq:p-z-def}. As stated in the
  proof of \cref{lem:I-Ptau-S1g}, we have
  \[
    p = \mathcal E_\mathbb F S_1 g,
  \]
  so that by \cref{eq:trE} we get
  \[ 
    u = \mathcal P_{[u_*,u^*]}
    \big( -\nu^{-1} \operatorname{tr}p \big).
  \]
  From the fact $ p \in L_\mathbb F^2(\Omega;C([0,T];\dot H^1)) $, we then
  conclude that
  \[
    u \in L_\mathbb F^2(\Omega;C([0,T];L^2(\Gamma))).
  \]
  Moreover,
  \begin{align*} 
    & \sum_{j=0}^{J-1} \ssnm{u - u(t_j)}_{
      L^2(t_j,t_{j+1};L^2(\Gamma))
    }^2 \\
    ={}&
    \sum_{j=0}^{J-1} \int_{t_j}^{t_{j+1}} \ssnm{
      \mathcal P_{[u_*,u^*]} (- \nu^{-1} \operatorname{tr}p(t)) -
      \mathcal P_{[u_*,u^*]} (- \nu^{-1} \operatorname{tr}p(t_j))
    }_{L^2(\Gamma)}^2 \, \mathrm{d}t \\
    \leqslant{} &
    \sum_{j=0}^{J-1} \int_{t_j}^{t_{j+1}} \ssnm{
      (p(t) - p(t_j)) / \nu
    }_{L^2(\Gamma)}^2 \, \mathrm{d}t \\
    \lesssim{} &
    \tau \nu^{-2} \ssnm{g}_{L^2(0,T;\dot H^0)}^2,
  \end{align*}
  by \cref{eq:p-pj2} and the trace inequality
  \[
    \nm{v}_{L^2(\Gamma)} \lesssim
    \nm{v}_{\dot H^\gamma} \quad \forall v \in \dot H^\gamma.
  \]
  Therefore, \cref{eq:I-Ptau-u} follows from the inequality
  \[ 
    \ssnm{(I-\mathcal P_\tau)u}_{
      L^2(0,T;L^2(\Gamma))
    } \leqslant \Big(
      \sum_{j=0}^{J-1} \ssnm{u - u(t_j)}_{
        L^2(t_j,t_{j+1};L^2(\Gamma))
      }^2
    \Big)^{1/2}.
  \]
  This completes the proof.
\end{proof}



Finally, we will use the argument in the numerical analysis of optimal control
problems with PDE constraints (cf.~\cite[Theorem 3.4]{Hinze2009}) to prove
\cref{thm:conv}.

\medskip\noindent{\bf Proof of \cref{thm:conv}}. In this proof, $ C $ is a
  positive constant depending only on $ u_* $, $ u^* $, $ y_d $, $ \epsilon $, $
  \sigma $, $ \nu $, $ \mathcal O $, $ T $ and the regularity parameters of $
  \mathcal K_h $, and its value may differ in different places. Let $ \bar P $
  be defined by \cref{eq:barP-def}, and define
  \begin{equation} 
    \label{eq:P-def}
    P := S_1^{h,\tau}(S_0\mathcal R\bar u + G - y_d).
  \end{equation}
  By \cref{eq:S0-conv} and the fact $ \bar u \in U_\text{ad} $, we obtain
  \begin{equation} 
    \label{eq:S0-calS0-u}
    \ssnm{(S_0 - S_0^{h,\tau})\mathcal R\bar u}_{
      L^2(0,T;\dot H^0)
    } \leqslant C ( \tau^{3/4-\epsilon} + h^{3/2-\epsilon} ).
  \end{equation}
  By \cref{eq:S0R-C,eq:G-regu} we have
  \[ 
    \ssnm{S_0 \mathcal R \bar u + G - y_d}_{L^2(0,T;\dot H^0)}
    \leqslant C,
  \]
  so that \cref{eq:S1-calS1} implies
  \begin{equation} 
    \label{eq:barp-P}
    \Big(
      \sum_{j=0}^{J-1} \ssnm{\bar p - P_{j+1}}_{
        L^2(t_j,t_{j+1};L^2(\Gamma))
      }^2
    \Big)^{1/2} \leqslant
    C( \tau^{3/4-\epsilon} + h^{3/2-\epsilon} ).
  \end{equation}
  By the definition of $ \mathcal P_\tau $, we have the following two evident
  equalities:
  \begin{small}
  \begin{align}
    & \int_0^T \dual{\bar U, \mathcal P_\tau \bar u} \, \mathrm{d}t =
    \int_0^T \dual{\bar U, \bar u} \, \mathrm{d}t,
    \label{eq:U-Ptauu-u} \\
    & \ssnm{\bar u - \bar U}_{L^2(0,T;L^2(\Gamma))}^2 =
    \ssnm{\bar u - \mathcal P_\tau u}_{L^2(0,T;L^2(\Gamma))}^2 +
    \ssnm{\bar U - \mathcal P_\tau u}_{L^2(0,T;L^2(\Gamma))}^2.
    \label{eq:u-U-ptauu}
  \end{align}
  \end{small}
  Since \cref{eq:conv-2} follows from
  \cref{eq:conv-1,lem:I-Ptau-S1g,lem:I-Ptau-u}, we only need to prove
  \cref{eq:conv-1}. The rest of the proof is
  divided into the following three steps.

  {\it Step 1.} Let us prove that
  \begin{small}
  \begin{equation}
    \label{eq:700}
    \nu \ssnm{\bar u - \bar U}_{
      L^2(0,T;L^2(\Gamma))
    }^2 + \frac12 \ssnm{
      S_0\mathcal R\bar u - S_0^{h,\tau}\mathcal R\bar U
    }_{
      L^2(0,T;\dot H^0)
    }^2 \leqslant I_0 + I_1 + I_2,
  \end{equation}
  \end{small}
  where
  \begin{align} 
    I_0 &:= \frac12 \ssnm{
      S_0\mathcal R\bar u - S_0^{h,\tau}\mathcal R\mathcal P_\tau \bar u
    }_{
      L^2(0,T;\dot H^0)
    }^2, \label{eq:I0} \\
    I_1 &:= \int_0^T [
      G - G_{h,\tau},
      S_0^{h,\tau} \mathcal R(\bar U - \mathcal P_\tau \bar u)
    ] \, \mathrm{d}t,  \label{eq:I1} \\
    I_2 & := \sum_{j=0}^{J-1}
    \int_{t_j}^{t_{j+1}} \dual{
      \bar p - \mathbb E_{t_j} P_{j+1}, \bar U - \bar u
    } \, \mathrm{d}t. \label{eq:I2}
  \end{align}
  Inserting $ u := \bar U $ into \eqref{eq:optim-cond} yields
  \begin{align*}
    \int_0^T \dual{
      \bar p + \nu \bar u,
      \bar U - \bar u
    } \, \mathrm{d}t \geqslant 0,
  \end{align*}
  which implies
  \begin{equation} 
    \label{eq:800}
    \nu \int_0^T \dual{\bar u, \bar u - \bar U} \, \mathrm{d}t
    \leqslant \int_0^T \dual{\bar p, \bar U - \bar u} \, \mathrm{d}t.
  \end{equation}
  Inserting $ U : = \mathcal P_\tau \bar u $ into \cref{eq:discr-optim} gives
  \begin{align*} 
    \sum_{j=0}^{J-1} \int_{t_j}^{t_{j+1}}
    \Dual{
      \bar P_{j+1} + \nu \bar U, \,
      \mathcal P_\tau \bar u - \bar U
    } \, \mathrm{d}t \geqslant 0,
  \end{align*}
  which, together with \cref{eq:U-Ptauu-u}, implies
  \begin{equation} 
    \label{eq:801}
    -\nu \int_0^T \dual{\bar U,\bar u - \bar U} \, \mathrm{d}t
    \leqslant \sum_{j=0}^{J-1} \int_{t_j}^{t_{j+1}}
    \dual{\bar P_{j+1}, \mathcal P_\tau \bar u - \bar U}
    \, \mathrm{d}t.
  \end{equation}
  From \cref{eq:800,eq:801} we conclude that
  \begin{small}
  \begin{align*} 
    & \nu \ssnm{\bar u - \bar U}_{
      L^2(0,T; L^2(\Gamma))
    }^2 \\
    \leqslant{} &
    \int_0^T \dual{
      \bar p, \bar U - \bar u
    } \mathrm{d}t +
    \sum_{j=0}^{J-1} \int_{t_j}^{t_{j+1}} \dual{
      \bar P_{j+1}, \mathcal P_\tau \bar u - \bar U
    } \mathrm{d}t \\
    ={} &
    \sum_{j=0}^{J-1} \int_{t_j}^{t_{j+1}}
    \dual{
      \bar p - \mathbb E_{t_j} P_{j+1}, \bar U - \bar u
    } \, \mathrm{d}t + \sum_{j=0}^{J-1}
    \int_{t_j}^{t_{j+1}} \dual{
      \mathbb E_{t_j} P_{j+1}, \bar U - \bar u
    } \, \mathrm{d}t \\
    & \qquad \qquad {} +
    \sum_{j=0}^{J-1} \int_{t_j}^{t_{j+1}}
    \dual{\bar P_{j+1}, \mathcal P_\tau \bar u - \bar U}
    \, \mathrm{d}t \\
    ={} &
    I_2 \!+\! \sum_{j=0}^{J-1} \! \int_{t_j}^{t_{j+1}} \!
    \dual{\mathbb E_{t_j} P_{j+1}, \bar U \!-\! \bar u} \mathrm{d}t +
    \sum_{j=0}^{J-1} \! \int_{t_j}^{t_{j+1}} \!
    \dual{\bar P_{j+1}, \mathcal P_\tau \bar u \!-\! \bar U}
    \mathrm{d}t \quad\text{(by \cref{eq:I2})} \\
    ={} &
    I_2 + \sum_{j=0}^{J-1}
    \int_{t_j}^{t_{j+1}} \dual{
      P_{j+1} - \bar P_{j+1},
      \bar U - \mathcal P_\tau \bar u
    } \, \mathrm{d}t,
  \end{align*}
  \end{small}
  by the definition of $ \mathcal P_\tau $.  Since
  \begin{small}
  \begin{align*} 
    & \sum_{j=0}^{J-1} \int_{t_j}^{t_{j+1}}
    \dual{
      P_{j+1} - \bar P_{j+1}, \bar U - \mathcal P_\tau \bar u
    } \, \mathrm{d}t \notag \\
    ={} &
    \sum_{j=0}^{J-1} \int_{t_j}^{t_{j+1}}
    \dual{
      \big(
        S_1^{h,\tau}(
        S_0\mathcal R\bar u -
        S_0^{h,\tau}\mathcal R \bar U + G - G_{h,\tau}
        )
      \big)_{j+1}, \bar U - \mathcal P_\tau\bar u
    } \, \mathrm{d}t
    \quad\text{(by \cref{eq:barP-def,eq:P-def})} \notag \\
    ={} &
    \int_0^T [
      S_0\mathcal R\bar u - S_0^{h,\tau}\mathcal R\bar U +
      G - G_{h,\tau},
      S_0^{h,\tau} \mathcal R(\bar U - \mathcal P_\tau\bar u)
    ] \, \mathrm{d}t
    \quad\text{(by \cref{eq:S1htau-S2htau})} \\
    ={} &
    I_1 + \int_0^T [
      S_0\mathcal R\bar u - S_0^{h,\tau}\mathcal R\bar U,
      S_0^{h,\tau} \mathcal R(\bar U - \mathcal P_\tau\bar u)
    ] \, \mathrm{d}t \quad\text{(by \cref{eq:I1}),}
  \end{align*}
  \end{small}
  it follows that
  \begin{equation} 
    \label{eq:120}
    \nu \ssnm{\bar u - \bar U}_{L^2(0,T;L^2(\Gamma))}^2
    \leqslant
    I_1 + I_2 + \int_0^T [
      S_0 \mathcal R\bar u - S_0^{h,\tau} \mathcal R\bar U,
      S_0^{h,\tau} \mathcal R(\bar U - \mathcal P_\tau \bar u)
    ] \, \mathrm{d}t.
  \end{equation}
  We also have
  \begin{align*} 
    & \int_0^T [
      S_0 \mathcal R\bar u - S_0^{h,\tau} \mathcal R\bar U,
      S_0^{h,\tau} \mathcal R(\bar U - \mathcal P_\tau \bar u)
    ] \, \mathrm{d}t \\
    ={} &
    -\ssnm{S_0 \mathcal R\bar u - S_0^{h,\tau} \mathcal R\bar U}_{
      L^2(0,T;\dot H^0)
    }^2 + \int_0^T [
      S_0\mathcal R\bar u - S_0^{h,\tau}\mathcal R \bar U,
      S_0 \mathcal R\bar u - S_0^{h,\tau} \mathcal R \mathcal P_\tau \bar u
    ] \, \mathrm{d}t \\
    \leqslant{} &
    -\frac12 \ssnm{S_0\mathcal R\bar u - S_0^{h,\tau}\mathcal R\bar U}_{
      L^2(0,T;\dot H^0)
    }^2 + \frac12 \ssnm{
      S_0\mathcal R\bar u - S_0^{h,\tau}\mathcal R\mathcal P_\tau \bar u
    }_{
      L^2(0,T;\dot H^0)
    }^2 \\
    ={} &
    -\frac12 \ssnm{S_0\mathcal R\bar u - S_0^{h,\tau}\mathcal R\bar U}_{
      L^2(0,T;\dot H^0)
    }^2 + I_0 \quad\text{(by \cref{eq:I0}).}
  \end{align*}
  Inserting the above inequality into \cref{eq:120} yields
  \[ 
    \nu \ssnm{\bar u - \bar U}_{L^2(0,T;\dot H^0)}^2
    \leqslant I_1 + I_2 - \frac12
    \ssnm{S_0\mathcal R\bar u - S_0^{h,\tau}\mathcal R\bar U}_{
      L^2(0,T;\dot H^0)
    }^2 + I_0,
  \]
  and then a simple calculation proves \cref{eq:700}.

  {\it Step 2}. Let us estimate $ I_0 $, $ I_1 $ and $ I_2 $.  For $ I_0 $, we
  have
  \begin{align*} 
    I_0 & = \frac12 \ssnm{
      (S_0 - S_0^{h,\tau}) \mathcal R\bar u +
      S_0^{h,\tau}\mathcal R(I-\mathcal P_\tau)\bar u
    }_{L^2(0,T;\dot H^0)}^2 \quad\text{(by \cref{eq:I0})} \\
    & \leqslant \ssnm{(S_0 - S_0^{h,\tau})\mathcal R\bar u}_{
      L^2(0,T;\dot H^0)
    }^2 + \ssnm{S_0^{h,\tau}\mathcal R(I-\mathcal P_\tau)\bar u}_{
      L^2(0,T;\dot H^0)
    }^2 \\
    & \leqslant
    \ssnm{(S_0 - S_0^{h,\tau})\mathcal R\bar u}_{
      L^2(0,T;\dot H^0)
    }^2 + C \ssnm{(I-\mathcal P_\tau)\bar u}_{
      L^2(0,T;L^2(\Gamma))
    }^2 \quad\text{(by \cref{eq:S0-stab})},
  \end{align*}
  so that \cref{eq:S0-calS0-u} implies
  \begin{equation}
    \label{eq:I0-esti}
    I_0 \leqslant C(\tau^{3/4-\epsilon} + h^{3/2-\epsilon})^2 +
    C \ssnm{(I-\mathcal P_\tau)\bar u}_{
      L^2(0,T;L^2(\Gamma))
    }^2.
  \end{equation}
  For $ I_1 $, we have
  \begin{align*} 
    I_1 & \leqslant
    \ssnm{G - G_{h,\tau}}_{
      L^2(0,T;\dot H^0)
    } \ssnm{S_0^{h,\tau} \mathcal R(\bar U - \mathcal P_\tau\bar u)}_{
      L^2(0,T;\dot H^0)
    } \quad\text{(by \cref{eq:I1})} \\
    & \leqslant
    C\ssnm{G - G_{h,\tau}}_{
      L^2(0,T;\dot H^0)
    } \ssnm{\bar U - \mathcal P_\tau\bar u}_{
      L^2(0,T;L^2(\Gamma))
    } \quad\text{(by \cref{eq:S0-stab})} \\
    & \leqslant C
    \ssnm{G - G_{h,\tau}}_{
      L^2(0,T;\dot H^0)
    } \ssnm{\bar u - \bar U}_{
      L^2(0,T;L^2(\Gamma))
    } \quad\text{(by \cref{eq:u-U-ptauu})},
  \end{align*}
  so that by \cref{eq:G-Ghtau} we obtain
  \begin{equation}
    \label{eq:I1-esti}
    I_1 \leqslant C (\tau^{1/4-\epsilon} + h^{1/2-\epsilon})
    \ssnm{\bar u - \bar U}_{L^2(0,T;L^2(\Gamma))}.
  \end{equation}
  Now let us estimate $ I_2 $. By definition, we have
  \[
    (\mathcal P_\tau \bar p)(t_j) =
    (\mathcal P_\tau \mathcal E_\mathbb F \bar p)(t_j),
    \quad \forall 0 \leqslant j < J,
  \]
  and so
  \begin{equation}
    \label{eq:I2-1}
    \sum_{j=0}^{J-1}
    \ssnm{
      \mathcal E_\mathbb F\bar p - (\mathcal P_\tau \bar p)(t_j)
    }_{L^2(t_j,t_{j+1};L^2(\Gamma))}^2 =
    \ssnm{(I-\mathcal P_\tau)\mathcal E_\mathbb F \bar p}_{
      L^2(0,T;L^2(\Gamma))
    }^2.
  \end{equation}
  We also have
  \begin{align*}
    & \sum_{j=0}^{J-1} \ssnm{
      (\mathcal P_\tau \bar p)(t_j) -
      \mathbb E_{t_j} P_{j+1}
    }_{L^2(t_j,t_{j+1};L^2(\Gamma))}^2 \\
    ={} &
    \sum_{j=0}^{J-1} \tau \ssnm{
      (\mathcal P_\tau \bar p)(t_j) -
      \mathbb E_{t_j} P_{j+1}
    }_{L^2(\Gamma)}^2 \\
    ={} &
    \sum_{j=0}^{J-1} \tau^{-1} \ssnm{
      \mathbb E_{t_j} \int_{t_j}^{t_{j+1}}
      \bar p(t) - P_{j+1} \, \mathrm{d}t
    }_{L^2(\Gamma)}^2 \\
    \leqslant{} & \sum_{j=0}^{J-1} \tau^{-1} \ssnm{
      \int_{t_j}^{t_{j+1}}
      \bar p(t) - P_{j+1} \, \mathrm{d}t
    }_{L^2(\Gamma)}^2 \\
    \leqslant{} &
    \sum_{j=0}^{J-1} \ssnm{
      \bar p - P_{j+1}
    }_{L^2(t_j,t_{j+1};L^2(\Gamma))}^2.
  \end{align*}
  Combing the above estimate with \cref{eq:I2-1} yields
  \begin{small}
  \begin{align*}
    & \Big(
      \sum_{j=0}^{J-1} \ssnm{
        \mathcal E_\mathbb F\bar p - \mathbb E_{t_j} P_{j+1}
      }_{L^2(t_j,t_{j+1};L^2(\Gamma))}^2
    \Big)^{1/2} \\
    \leqslant{} &
    \Big(
      \sum_{j=0}^{J-1} \ssnm{
        \mathcal E_\mathbb F\bar p - (\mathcal P_\tau\bar p)(t_j)
      }_{L^2(t_j,t_{j+1};L^2(\Gamma))}^2
    \Big)^{1/2} + {} \\
    & \quad \Big(
      \sum_{j=0}^{J-1} \ssnm{
        (\mathcal P_\tau \bar p)(t_j) -
        \mathbb E_{t_j} P_{j+1}
      }_{
        L^2(t_j,t_{j+1};L^2(\Gamma))
      }^2
    \Big)^{1/2} \\
    \leqslant{} &
    \ssnm{(I-\mathcal P_\tau)\mathcal E_\mathbb F \bar p}_{L^2(0,T;L^2(\Gamma))} +
    \Big(
      \sum_{j=0}^{J-1} \ssnm{\bar p - P_{j+1}}_{
        L^2(t_j,t_{j+1};L^2(\Gamma))
      }^2
    \Big)^{1/2}.
  \end{align*}
  \end{small}
  Hence,
  \begin{align*}
    I_2 &= \sum_{j=0}^{J-1}
    \int_{t_j}^{t_{j+1}} \dual{
      \mathcal E_\mathbb F \operatorname{tr} \bar p -
      \mathbb E_{t_j} P_{j+1},
      \bar U - \bar u
    } \, \mathrm{d}t \quad\text{(by \cref{eq:I2})} \\
    & = \sum_{j=0}^{J-1}
    \int_{t_j}^{t_{j+1}} \dual{
      \mathcal E_\mathbb F \bar p -
      \mathbb E_{t_j} P_{j+1},
      \bar U - \bar u
    } \, \mathrm{d}t \quad\text{(by \cref{eq:trE})} \\
    &\leqslant
    \Big(
      \sum_{j=0}^{J-1} \ssnm{
        \mathcal E_\mathbb F\bar p - \mathbb E_{t_j} P_{j+1}
      }_{
        L^2(t_j,t_{j+1};L^2(\Gamma))
      }^2
    \Big)^{1/2} \ssnm{\bar u - \bar U}_{L^2(0,T;L^2(\Gamma))} \\
    & \leqslant
    \Big(
      \ssnm{(I-\mathcal P_\tau)\mathcal E_\mathbb F\bar p}_{
        L^2(0,T;L^2(\Gamma))
      } + \Big(
        \sum_{j=0}^{J-1} \ssnm{\bar p - P_{j+1}}_{
          L^2(t_j,t_{j+1};L^2(\Gamma))
        }^2
      \Big)^{1/2}
    \Big) \\
    & \qquad\qquad {} \times \ssnm{\bar u - \bar U}_{L^2(0,T;L^2(\Gamma))},
  \end{align*}
  which, together with \cref{eq:barp-P}, implies
  \begin{small}
  \begin{equation}
    \label{eq:I2-esti}
    I_2 \leqslant
    \big(
      C(\tau^{3/4-\epsilon} + h^{3/2-\epsilon}) +
      \ssnm{(I-\mathcal P_\tau)\mathcal E_\mathbb F \bar p}_{
        L^2(0,T;L^2(\Gamma))
      }
    \big) \ssnm{\bar u - \bar U}_{L^2(0,T;L^2(\Gamma))}.
  \end{equation}
  \end{small}

  {\it Step 3.} Combining \cref{eq:700,eq:I0-esti,eq:I1-esti,eq:I2-esti}, we
  conclude that
  \begin{small} 
  \begin{align*}
    & \nu\ssnm{\bar u - \bar U}_{L^2(0,T;L^2(\Gamma))}^2 +
    \ssnm{S_0\mathcal R\bar u - S_0^{h,\tau}\mathcal R\bar U}_{
      L^2(0,T;\dot H^0)
    }^2 \\
    \leqslant{} &
    C \Big( (\tau^{3/4-\epsilon} + h^{3/2-\epsilon})^2 +
    \ssnm{(I-\mathcal P_\tau)\bar u}_{L^2(0,T;L^2(\Gamma))}^2 \\
    & \qquad {} +
    (\tau^{1/4-\epsilon} + h^{1/2-\epsilon})
    \ssnm{\bar u - \bar U}_{L^2(0,T;L^2(\Gamma))} \\
    & \qquad {} +
    \ssnm{(I-\mathcal P_\tau)\mathcal E_\mathbb F\bar p}_{L^2(0,T;L^2(\Gamma))}
    \ssnm{\bar u-\bar U}_{L^2(0,T;L^2(\Gamma))} \Big).
  \end{align*}
  \end{small}
  Using the Young's inequality with $ \epsilon $, we then obtain
  \begin{small}
  \begin{align*}
    & \ssnm{\bar u - \bar U}_{L^2(0,T;L^2(\Gamma))} +
    \ssnm{
      S_0\mathcal R \bar u - S_0^{h,\tau} \mathcal R\bar U
    }_{L^2(0,T;\dot H^0)} \\
    \leqslant{} &
    C\Big(
      \tau^{1/4-\epsilon} + h^{1/2-\epsilon} +
      \ssnm{(I-\mathcal P_\tau)\bar u}_{L^2(0,T;L^2(\Gamma))} +
      \ssnm{(I-\mathcal P_\tau)\mathcal E_\mathbb F\bar p}_{L^2(0,T;L^2(\Gamma))}
    \Big).
  \end{align*}
  \end{small}
  namely inequality \cref{eq:conv-1}. This completes the proof.
\hfill\ensuremath{\blacksquare}

\section{Conclusion}
\label{sec:conclusion}
In this paper we analyzed the discretization of a Neumann boundary control
problem, where the state equation is a linear stochastic parabolic equation with
boundary noise. With rough data, we have established the convergence for general
filtration, and we have derived the convergence rate $ O(\tau^{1/4-\epsilon} +
h^{1/2-\epsilon}) $ for the natural filtration of the $ Q $-Wiener process,
where $ 0 < \epsilon < 1/4 $ can be arbitrarily small.

The numerical analysis in this paper can be extended to the distributed optimal
control problems governed by linear stochastic parabolic equations with additive
noise and general filtration. For the case that the noise is multiplicative or
the coefficient before the $ Q $-Wiener process contains the control variable,
there will be essential difficulty in the numerical analysis, especially for
general filtration. To solve these issues is our ongoing work.

\Acknowledgements{This work was supported by National Natural Science Foundation
of China (Grant No. 11901410) and the Fundamental Research Funds for the Central
Universities in China (Grant No. 2020SCU12063).}




\begin{appendix}
\section{An interpolation space result}
Throughout this section, let $ (\cdot,\cdot)_{1-\theta,2} $, $ 0 < \theta < 1 $,
denote the interpolation space defined by the $ K $-method
(cf.~\cite[Chapter~1]{Lunardi2018}).
\begin{lemma}
  \label{lem:auxi}
  The space
  \[
    L_\mathbb F^2\Big(
      \Omega;L^2\big(
        0,T;\mathcal L_2\big(L_\lambda^2, \dot H^{-\theta}\big)
      \big)
    \Big)
  \]
  is continuously embedded into
  \[
    \Big(
      L_\mathbb F^2\big(
        \Omega;L^2\big(0,T;\mathcal L_2(L_\lambda^2,\dot H^{-1})\big)
      \big), \,
      L_\mathbb F^2\big(
        \Omega;L^2\big(0,T;\mathcal L_2(L_\lambda^2, \dot H^0)\big)
      \big)
    \Big)_{1-\theta,2},
  \]
  where $ 0 < \theta < 1 $.
\end{lemma}
\begin{proof}
  For any $ B \in \mathcal L_2\big( L_\lambda^2, \, (\dot H^{-1}, \dot
  H^0)_{1-\theta,2} \big) $, by definition we have
  \begin{align*} 
  & \nm{B}_{
    \big(
      \mathcal L_2(L_\lambda^2, \dot H^{-1}), \,
      \mathcal L_2(L_\lambda^2, \dot H^0)
    \big)_{1-\theta,2}
  }^2 \\
  ={} &
  \int_0^\infty t^{2\theta-3} \left(
    \inf_{
      \substack{
        B = B_0 + B_1 \\
        B_0 \in \mathcal L_2(L_\lambda^2, \dot H^{-1}) \\
        B_1 \in \mathcal L_2(L_\lambda^2, \dot H^0)
      }
      } \Big(
      \nm{B_0}_{
        \mathcal L_2(L_\lambda^2, \dot H^{-1})
        } + t \nm{B_1}_{
        \mathcal L_2(L_\lambda^2, \dot H^0)
      }
    \Big)
  \right)^2 \, \mathrm{d}t \\
  \leqslant{} &
  2\int_0^\infty t^{2\theta-3} \inf_{
    \substack{
      B = B_0 + B_1 \\
      B_0 \in \mathcal L_2(L_\lambda^2, \dot H^{-1}) \\
      B_1 \in \mathcal L_2(L_\lambda^2, \dot H^0)
    }
    } \Big(
    \nm{B_0}_{
      \mathcal L_2(L_\lambda^2, \dot H^{-1})
      }^2 + t^2 \nm{B_1}_{
      \mathcal L_2(L_\lambda^2, \dot H^0)
    }^2
  \Big) \, \mathrm{d}t \\
  ={} &
  2\int_0^\infty t^{2\theta-3} \inf_{
    \substack{
      B = B_0 + B_1 \\
      B_0 \in \mathcal L_2(L_\lambda^2, \dot H^{-1}) \\
      B_1 \in \mathcal L_2(L_\lambda^2, \dot H^0)
    }
    } \sum_{n=0}^\infty \lambda_n \Big(
    \nm{B_0\phi_n}_{ \dot H^{-1} }^2 +
    t^2 \nm{B_1\phi_n}_{ \dot H^0 }^2
  \Big) \, \mathrm{d}t \\
  ={} & 2 \int_0^\infty t^{2\theta-3}
  \sum_{n=0}^\infty \lambda_n \inf_{
    \substack{
      B\phi_n = v_0 + v_1 \\
      v_0 \in \dot H^{-1} \\
      v_1 \in \dot H^0
    }
  } \Big( \nm{v_0}_{\dot H^{-1}}^2 + t^2 \nm{v_1}_{\dot H^0}^2 \Big) \, \mathrm{d}t
  \quad\text{(by \cref{lem:tmp})} \\
  \leqslant{} &
  2\int_0^\infty t^{2\theta-3}
  \sum_{n=0}^\infty \lambda_n \left(
    \inf_{
      \substack{
        B\phi_n = v_0 + v_1 \\
        v_0 \in \dot H^{-1} \\
        v_1 \in \dot H^0
      }
    } \Big( \nm{v_0}_{\dot H^{-1}} + t \nm{v_1}_{\dot H^0} \Big)
  \right)^2 \, \mathrm{d}t \\
  ={} &
  2\sum_{n=0}^\infty \lambda_n \int_0^\infty t^{2\theta-3} \left(
    \inf_{
      \substack{
        B\phi_n = v_0 + v_1 \\
        v_0 \in \dot H^{-1} \\
        v_1 \in \dot H^0
      }
    } \Big( \nm{v_0}_{\dot H^{-1}} + t \nm{v_1}_{\dot H^0} \Big)
  \right)^2 \, \mathrm{d}t \\
  ={} &
  2\sum_{n=0}^\infty \lambda_n \nm{B\phi_n}_{(\dot H^{-1}, \dot H^0)_{1-\theta,2}}^2 \\
  ={} &
  2\nm{B}_{
    \mathcal L_2(L_\lambda^2, \, (\dot H^{-1}, \dot H^0)_{1-\theta, 2})
  }^2.
\end{align*}
This implies that $ \mathcal L_2(L_\lambda^2, (\dot H^{-1}, \dot H^0)
_{1-\theta,2}) $ is continuously embedded into
\[
  \big(
    \mathcal L_2(L_\lambda^2, \dot H^{-1}), \,
    \mathcal L_2(L_\lambda^2, \dot H^0)
  \big)_{1-\theta,2}.
\]
Since $ \dot H^{-\theta} = (\dot H^{-1}, \dot H^0)_{1-\theta,2} $ with
equivalent norms, we readily conclude that $ \mathcal L_2(L_\lambda^2,
\dot H^{-\theta}) $ is continuously embedded into
\[
  \big(
    \mathcal L_2(L_\lambda^2, \dot H^{-1}), \,
    \mathcal L_2(L_\lambda^2, \dot H^0)
  \big)_{1-\theta,2}.
\]
Therefore, the desired claim follows from the result that
(cf.~\cite[p.~38]{Lunardi2018})
\[
  L_\mathbb F^2\bigg(
    \Omega;L^2\Big(
      0,T; \big(
        \mathcal L_2(L_\lambda^2, \dot H^{-1}), \,
        \mathcal L_2(L_\lambda^2, \dot H^0)
      \big)_{1-\theta,2}
    \Big)
  \bigg)
\]
is identical to
\[
  \bigg(
    L_\mathbb F^2\Big(
      \Omega;L^2\big(0,T;\mathcal L_2(L_\lambda^2,\dot H^{-1})\big)
    \Big), \,
    L_\mathbb F^2\Big(
      \Omega;L^2\big(0,T;\mathcal L_2(L_\lambda^2, \dot H^0)\big)
    \Big)
  \bigg)_{1-\theta,2},
\]
with equivalent norms. This completes the proof.
\end{proof}

\begin{lemma}
  \label{lem:tmp}
  For any $ t > 0 $ and $ B \in \mathcal L_2(L_\lambda^2, (\dot H^{-1}, \dot H^0)_{1-\theta,2})
  $ with $ 0 < \theta < 1 $, we have
  \begin{equation}
    \begin{aligned}
      & \inf_{
        \substack{
          B = B_0 + B_1 \\
          B_0 \in \mathcal L_2(L_\lambda^2, \dot H^{-1}) \\
          B_1 \in \mathcal L_2(L_\lambda^2, \dot H^0)
        }
        } \sum_{n=0}^\infty \lambda_n \Big(
        \nm{B_0\phi_n}_{ \dot H^{-1} }^2 + t^2 \nm{B_1\phi_n}_{ \dot H^0 }^2
      \Big) \\
      ={} &
      \sum_{n=0}^\infty \lambda_n
      \inf_{\substack{B\phi_n = v_0 + v_1 \\ v_0 \in \dot H^{-1} \\ v_1 \in \dot
        H^0}} \Big(
        \nm{v_0}_{\dot H^{-1}}^2 + t^2 \nm{v_1}_{\dot H^0}^2
      \Big).
    \end{aligned}
  \end{equation}
\end{lemma}
\begin{proof}
  By the evident inequality
  \begin{equation*}
    \begin{aligned}
      & \inf_{
        \substack{
          B = B_0 + B_1 \\
          B_0 \in \mathcal L_2(L_\lambda^2, \dot H^{-1}) \\
          B_1 \in \mathcal L_2(L_\lambda^2, \dot H^0)
        }
        } \sum_{n=0}^\infty \lambda_n \Big(
        \nm{B_0\phi_n}_{ \dot H^{-1} }^2 + t^2 \nm{B_1\phi_n}_{ \dot H^0 }^2
      \Big) \\
      \geqslant{} &
      \sum_{n=0}^\infty \lambda_n
      \inf_{\substack{B\phi_n = v_0 + v_1 \\ v_0 \in \dot H^{-1} \\ v_1 \in \dot
        H^0}} \Big(
        \nm{v_0}_{\dot H^{-1}}^2 + t^2 \nm{v_1}_{\dot H^0}^2
      \Big),
    \end{aligned}
  \end{equation*}
  it remains only to prove
  \begin{equation}
    \label{eq:tmp}
    \begin{aligned}
      & \inf_{
        \substack{
          B = B_0 + B_1 \\
          B_0 \in \mathcal L_2(L_\lambda^2, \dot H^{-1}) \\
          B_1 \in \mathcal L_2(L_\lambda^2, \dot H^0)
        }
        } \sum_{n=0}^\infty \lambda_n \Big(
        \nm{B_0\phi_n}_{ \dot H^{-1} }^2 + t^2 \nm{B_1\phi_n}_{ \dot H^0 }^2
      \Big) \\
      \leqslant{} &
      \sum_{n=0}^\infty \lambda_n
      \inf_{\substack{B\phi_n = v_0 + v_1 \\ v_0 \in \dot H^{-1} \\ v_1 \in \dot
        H^0}} \Big(
        \nm{v_0}_{\dot H^{-1}}^2 + t^2 \nm{v_1}_{\dot H^0}^2
      \Big).
    \end{aligned}
  \end{equation}
  To this end, we proceed as follows. Let $ \{v_{0,n} \mid n \in \mathbb N \}
  \subset \dot H^{-1} $ and $ \{v_{1,n}
    \mid n \in
  \mathbb N \} \subset \dot H^0 $ be arbitrary such that
  \[
    B\phi_n = v_{0,n} + v_{1,n}  \quad \forall n \in \mathbb N
  \]
  and that
  \begin{equation}
    \label{eq:love}
    \sum_{n=0}^\infty \lambda_n \big(
      \nm{v_{0,n}}_{\dot H^{-1}}^2 +
      t^2 \nm{v_{1,n}}_{\dot H^0}^2
    \big) < \infty.
  \end{equation}
  Define $ B_0 \in \mathcal L_2(L_\lambda^2, \dot H^{-1}) $ by
  \[
    B_0 \phi_n := v_{0,n} \quad \forall n \in \mathbb N,
  \]
  and define $ B_1 \in \mathcal L_2(L_\lambda^2, \dot H^0) $ by
  \[
    B_1 \phi_n := v_{1,n} \quad \forall n \in \mathbb N.
  \]
  By \cref{eq:love}, it is easily verified that the above $ B_0 $ and $ B_1 $
  are both well-defined. Clearly, we have
  \[
    B = B_0 + B_1
  \]
  and
  \[
    \sum_{n=0}^\infty \lambda_n \Big(
      \nm{B_0 \phi_n}_{\dot H^{-1}}^2 +
      t^2 \nm{B_1 \phi_n}_{\dot H^0}^2
    \Big) = \sum_{n=0}^\infty \lambda_n \Big(
    \nm{v_{0,n}}_{\dot H^{-1}}^2 + t^2 \nm{v_{1,n}}_{\dot H^0}^2
  \Big).
  \]
  Because of the arbitrary choices of $ \{v_{0,n}\mid n \in \mathbb N\} $ and $
  \{v_{1,n}\mid n \in \mathbb N\} $, we readily conclude \cref{eq:tmp}. This
  completes the proof.
\end{proof}

\end{appendix}

\end{document}